\documentclass[preprint,11pt]{elsarticle}

\usepackage{amsmath,amssymb,amsfonts,amsthm,mathtools}
\usepackage[shortlabels]{enumitem}
\usepackage{booktabs,longtable,array}
\usepackage{xcolor}
\usepackage{hyperref}

\hypersetup{
  colorlinks=true,
  linkcolor=blue!55!black,
  citecolor=blue!55!black,
  urlcolor=blue!65!black
}
\usepackage[capitalise,noabbrev]{cleveref}

\allowdisplaybreaks
\theoremstyle{plain}
\newtheorem{theorem}{Theorem}[section]
\newtheorem{proposition}[theorem]{Proposition}
\newtheorem{lemma}[theorem]{Lemma}
\newtheorem{corollary}[theorem]{Corollary}

\theoremstyle{definition}

\newtheorem{example}[theorem]{Example}

\theoremstyle{remark}
\newtheorem{remark}[theorem]{Remark}

\newcommand{\A}{\mathbb A}
\newcommand{\cA}{\mathcal A}
\newcommand{\Pone}{\mathbb P^1}
\newcommand{\PP}{\mathbb P}
\newcommand{\Q}{\mathbb Q}
\newcommand{\Z}{\mathbb Z}
\newcommand{\cM}{\mathcal M}
\newcommand{\cR}{\mathcal R}
\newcommand{\cS}{\mathcal S}

\newcommand{\cI}{\mathcal I}
\newcommand{\cJ}{\mathcal J}

\newcommand{\cB}{\mathcal B}
\newcommand{\cG}{\mathcal G}
\newcommand{\cH}{\mathcal H}
\newcommand{\cN}{\mathcal N}
\newcommand{\cT}{\mathcal T}
\newcommand{\bbG}{\mathbb G}
\newcommand{\N}{\mathbb N}
\newcommand{\F}{\mathbb F}
\newcommand{\bone}{\mathbf 1}
\newcommand{\Gor}{\mathrm{Gor}}

\DeclareMathOperator{\PGL}{PGL}
\DeclareMathOperator{\Proj}{Proj}
\DeclareMathOperator{\Rat}{Rat}
\DeclareMathOperator{\Res}{Res}
\DeclareMathOperator{\SL}{SL}
\DeclareMathOperator{\Spec}{Spec}
\DeclareMathOperator{\Sym}{Sym}
\DeclareMathOperator{\ord}{ord}
\DeclareMathOperator{\rank}{rank}
\DeclareMathOperator{\Frac}{Frac}
\DeclareMathOperator{\Hilb}{Hilb}
\DeclareMathOperator{\bideg}{bideg}
\DeclareMathOperator{\mult}{mult}
\DeclareMathOperator{\disc}{disc}
\DeclareMathOperator{\Cov}{Cov}
\DeclareMathOperator{\trdeg}{trdeg}
\DeclareMathOperator{\Span}{span}
\DeclareMathOperator{\Sub}{Sub}
\DeclareMathOperator{\tr}{tr}
\DeclareMathOperator{\diag}{diag}
\DeclareMathOperator{\Stab}{Stab}

\crefname{theorem}{Thm.}{theorems}
\crefname{proposition}{Prop.}{propositions}
\crefname{lemma}{Lem.}{lemmas}
\crefname{corollary}{Cor.}{corollaries}
\crefname{conjecture}{Conj.}{conjectures}
\crefname{definition}{Def.}{definitions}
\crefname{example}{Exa.}{examples}
\crefname{problem}{Prob.}{problems}
\crefname{remark}{Rem.}{remarks}
\crefname{equation}{Eq.}{Eqs.}

\crefformat{appendix}{#2#1#3}
\Crefformat{appendix}{#2#1#3}
\crefrangeformat{appendix}{#3#1#4 to #5#2#6}
\Crefrangeformat{appendix}{#3#1#4 to #5#2#6}
\crefmultiformat{appendix}{#2#1#3}{ and #2#1#3}{, #2#1#3}{, and #2#1#3}
\Crefmultiformat{appendix}{#2#1#3}{ and #2#1#3}{, #2#1#3}{, and #2#1#3}
\crefrangemultiformat{appendix}{...}
\Crefrangemultiformat{appendix}{...}

\usepackage{algorithm,algpseudocode}

\makeatletter
\newcommand{\sh@thm}[2]{%
  \expandafter\let\csname #1\endcsname\relax
  \expandafter\let\csname end#1\endcsname\relax
  \newtheorem{#1}{#2}%
  \expandafter\let\csname c@#1\endcsname\c@theorem
  \expandafter\def\csname the#1\endcsname{\thetheorem}%
}
\theoremstyle{plain}
\sh@thm{proposition}{Proposition}
\sh@thm{lemma}{Lemma}
\sh@thm{corollary}{Corollary}
\sh@thm{conjecture}{Conjecture}
\theoremstyle{definition}
\sh@thm{definition}{Definition}
\sh@thm{example}{Example}
\sh@thm{problem}{Problem}
\theoremstyle{remark}
\sh@thm{remark}{Remark}
\makeatother

\usepackage[margin=1in]{geometry}

\begin{document}

\begin{frontmatter}

\title{The invariant ring of degree-four rational maps on the projective line}

\author{T. Shaska}
\ead{shaska@oakland.edu}
\address{Department of Mathematics, Oakland University, Rochester, MI 48309, USA}

 \begin{abstract}
Let \(k\) be an algebraically closed field of characteristic zero. Degree-four rational maps on \(\PP^1\), up to conjugation, correspond to pairs of binary forms \((F,G)\in V_5\oplus V_3\). The associated invariant ring is the joint invariant ring \(\cR_{5,3}=k[V_5\oplus V_3]^{\SL_2}\).

We determine a minimal generating set for \(\cR_{5,3}\), consisting of fifty explicit joint transvectants of degrees at most \(18\). As consequences we describe the null cone of \(V_5\oplus V_3\), construct six explicit absolute invariants that generate the function field of the moduli space \(\cM_4^1\), and give an effective criterion for determining conjugacy of degree-four rational maps.
\end{abstract}


\end{frontmatter}

\section{Introduction}\label{sec:introduction}

Let $k$ be an algebraically closed field of characteristic zero.  A rational map $\phi:\PP^1\to\PP^1$ of degree $d\geq2$ is represented by a pair of degree-$d$ binary forms $\phi=[f_0:f_1]$ with $\Res(f_0,f_1)\neq0$, taken up to a common nonzero scalar.  Thus
\(
  \Rat_d^1=\PP(V_d\oplus V_d)\setminus V\bigl(\Res(f_0,f_1)\bigr)
\)
is an open subset of $\PP^{2d+1}$.  The group $\PGL_2(k)$ acts by conjugation, and the corresponding quotient is the moduli space $\cM_d^1$ of degree-$d$ rational maps; see \cite{Silverman1998,Silverman2012,West2015}.

The classical Clebsch--Gordan decomposition replaces this conjugation problem by one involving a pair of binary forms of different degrees.  Write $W=k^2$.  A pair $(f_0,f_1)$ is the coordinate expression of an element of $\Sym^d(W^\vee)\otimes W$, on which $\SL_2$ acts by substitution in the source together with the inverse linear action in the target; this is the action induced by conjugation of $\phi$.  The decomposition $\Sym^d(W^\vee)\otimes W\simeq V_{d+1}\oplus V_{d-1}$ is $\SL_2$-equivariant, and in the coordinates fixed above its two projections are
\begin{equation}\label{eq:associated-forms-intro}
  F=\cI_\phi=yf_0-xf_1\in V_{d+1},
  \qquad
  G=\cJ_\phi=\frac{\partial f_0}{\partial x}
             +\frac{\partial f_1}{\partial y}\in V_{d-1}.
\end{equation}
Thus the invariant theory of degree-$d$ rational maps is governed by the joint invariant ring $\cR_{d+1,d-1}=k[V_{d+1}\oplus V_{d-1}]^{\SL_2}$.
This paper studies the first case beyond rational cubics, namely $d=4$.  The corresponding representation is $V_5\oplus V_3$.  Our point of departure is the degree-three analysis of   \cite{2024-04}, together with the Hilbert--null-cone approach to pairs of binary cubics developed in \cite{2004-3,Krishnamoorthy2001}.  These examples suggest an important distinction which becomes essential in degree four: the polynomial invariant ring, the field of absolute invariants, and the ring of regular absolute invariants are different objects and require different methods.

These three objects have different sizes.  The representation $V_5\oplus V_3$ has dimension $10$, the invariant ring $\cR_{5,3}$ has Krull dimension $7$, and $\cM_4^1$ has dimension $6$.  Hence a homogeneous system of parameters for $\cR_{5,3}$ has seven members, while the field of absolute invariants has transcendence degree six.  Neither number is the number of algebra generators of $\cR_{5,3}$.

Our first main result determines the polynomial invariant ring $\cR_{5,3}=k[V_5\oplus V_3]^{\SL_2}$.  Using Gordan's theorem for joint covariants we obtain the unconditional generator bound $\beta(\cR_{5,3})\leq42$, and then prove that $\cR_{5,3}$ is minimally generated by fifty explicit joint invariants, all of degree at most $18$.  In particular, $\beta(\cR_{5,3})=18$.  This gives an explicit weighted-projective invariant model for degree-four rational maps, in which the fifty generators separate the conjugacy classes of rational maps.

A second main result concerns the structure of $\cR_{5,3}$ as a finite module over a polynomial subring.  We determine the null cone of $V_5\oplus V_3$: a pair $(F,G)$ is null exactly when $F$ and $G$ have a common zero of multiplicities at least $3$ and $2$, respectively.  From this we construct an explicit homogeneous system of parameters containing the rational-map resultant, and prove that $\cR_{5,3}$ is free of rank $604$ over the corresponding polynomial ring.  We also show that no homogeneous system of parameters of $\cR_{5,3}$ consists of bihomogeneous invariants.

The invariant field has a much simpler description.  On the dense open set where the cubic $G$ has three distinct roots, a projective change of variable normalizes those roots to $\{0,1,\infty\}$.  The remaining ambiguity is the anharmonic group $\Gamma\simeq S_3$ preserving $\{0,1,\infty\}$, acting on the normalized slice with a scalar correction, and the resulting six-dimensional representation is the regular representation of $S_3$.  From this model we construct six explicit absolute invariants $U_1,\ldots,U_6$ and prove that they generate the function field of moduli: $k(\cM_4^1)=k(U_1,\ldots,U_6)$.  In particular, $\cM_4^1$ is rational over $\Q$.  This gives an explicit invariant-theoretic proof in degree four of the rationality known in general from \cite{Levy2011}.

We finally pass from the invariant field to rings of regular absolute invariants.  On the chart where the cubic discriminant is nonzero, the residual $\Gamma$-action allows us to determine a minimal system of fourteen generators, all of which we realize as quotients of joint invariants of equal degree.  On the full resultant-open moduli space, the coordinate ring is the degree-zero part of $\cR_{5,3}$ localized at the resultant; its monomial structure is described by the Hilbert basis of the corresponding affine semigroup.  These constructions lead to explicit tests for conjugacy and also clarify why the absolute invariant ring cannot be generated merely from quotients of a homogeneous system of parameters.

The Clebsch--Gordan decomposition also retains a direct dynamical meaning.  The form $F$ is the fixed-point form of $\phi$: its zero divisor is the degree-$(d+1)$ fixed-point divisor of the map.  The multiplier at a simple fixed point is an explicit rational expression in $F$, $G$, and the first derivatives of $F$.  Thus the invariant theory of $V_5\oplus V_3$ simultaneously records the geometry of the fixed-point configuration and the conjugacy class of the rational map.

The paper is organized as follows.  \Cref{sec:cg} develops the Clebsch--Gordan coordinates, and \cref{sec:fixed-points} gives their dynamical interpretation.  \Cref{sec:covariants} recalls the required covariants and transvectants.  \Cref{sec-3} determines the minimal generators of $\cR_{5,3}$.  \Cref{sec-4} studies the null cone and homogeneous systems of parameters.  \Cref{sec-5} determines the field of absolute invariants and proves rationality, and \cref{section-6} treats the corresponding rings of regular absolute invariants.  The computational data are collected in the appendices.

\section{Preliminaries}\label{sec-2}

Throughout this paper we let $k$ be an algebraically closed field of characteristic zero and $\Pone$ denote the projective line over $k$.

\subsection{Rational maps and Clebsch--Gordan coordinates}\label{sec:cg}

For $n\geq 0$, let $V_n=\Sym^n(k^2)^\vee$ denote the space of binary forms of order $n$.  A degree-$d$ rational map is a pair $(f_0,f_1)\in V_d\oplus V_d$ with $\Res(f_0,f_1)\neq0$, taken up to a common nonzero scalar.  
For $\sigma=\left(\begin{smallmatrix} a&b\\ c&e\end{smallmatrix}\right)\in\SL_2(k)$ and $f\in V_n$, write $f^\sigma(x,y)=f(ax+by,\,cx+ey)$ for the substitution action.  Conjugation of $\varphi=[f_0:f_1]$ by $\sigma$ gives $\varphi^\sigma=\sigma^{-1}\varphi\sigma=[g_0:g_1]$ with
\begin{equation}\label{eq:conjugation-action}
  g_0=e\,f_0^\sigma-b\,f_1^\sigma,\qquad
  g_1=-c\,f_0^\sigma+a\,f_1^\sigma .
\end{equation}
This is the action of $\SL_2(k)$ on $V_d\oplus V_d$ used throughout. 

The correspondence between rational maps and pairs of binary forms was established in every degree $d\geq2$, see \cite[Lemmas~1 and 2]{2024-04} for details.  Define the linear map
\begin{equation}\label{eq:cg-map}
\begin{split}
  \Theta_d:V_d\oplus V_d   &	\longrightarrow V_{d+1}\oplus V_{d-1}, \\
 (f_0, f_1 ) & \longrightarrow       (F, G) = \left(  yf_0-xf_1,    (f_0)_x+(f_1)_y\right).
 \end{split} 
\end{equation}
The pair $(F,G)=\Theta_d(f_0,f_1)$ is called the Clebsch--Gordan pair of $\varphi$.  Then $\Theta_d(g_0,g_1)=(F^\sigma,G^\sigma)$, so $\Theta_d$ intertwines \cref{eq:conjugation-action} with simultaneous substitution on $(F,G)$, see \cite[Lemma~1]{2024-04} for the proof.

Euler's identity for the degree-$d$ forms $f_0,f_1$ gives $xG+F_y=(d+1)f_0$ and $yG-F_x=(d+1)f_1$.  Hence $\Theta_d$ is injective, and both sides have dimension $2d+2$, so $\Theta_d$ is invertible with inverse
\begin{equation}\label{eq:inverse-cg}
  f_0=\frac{xG+F_y}{d+1},\qquad
  f_1=\frac{yG-F_x}{d+1}.
\end{equation}
For $(F,G)\in V_{d+1}\oplus V_{d-1}$ with $(f_0,f_1)=\Theta_d^{-1}(F,G)$, define the rational-map resultant by
\begin{equation}\label{def:modular-resultant}
  \Delta_{F,G}=\Res(xG+F_y,\,yG-F_x)=(d+1)^{2d}\Res(f_0,f_1).
\end{equation}
The second equality holds because the two arguments are $(d+1)f_0$ and $(d+1)f_1$ and the resultant of two degree-$d$ forms is homogeneous of degree $d$ in the coefficients of each.  The invariant $\Delta_{F,G}$ is the modular resultant introduced in \cite[Definition~2]{2024-04}.  It is homogeneous of total degree $2d$ in the coefficients of $(F,G)$.  It is a nonzero multiple of $\Res(f_0,f_1)$, which is unchanged by \cref{eq:conjugation-action}, and therefore it is an $\SL_2$-invariant.  The projectivization of $\Theta_d$ identifies $\Rat_d^1$ with $\PP(V_{d+1}\oplus V_{d-1})\setminus V(\Delta_{F,G})$, see \cite[Lemma~2]{2024-04} for the proof.

For the remainder of the paper $d=4$.  Thus $F\in V_5$, $G\in V_3$, and
\begin{equation}\label{eq:quartic-reconstruction}
  f_0=\frac{xG+F_y}{5},\qquad
  f_1=\frac{yG-F_x}{5}.
\end{equation}
Every pair $(F,G)\in V_5\oplus V_3$ reconstructs a pair of quartics by \cref{eq:quartic-reconstruction}, and that pair defines a degree-four rational map precisely when $f_0$ and $f_1$ have no common root.  This requirement is the non-vanishing of a single $\SL_2$-invariant of $(F,G)$, namely
\( \Delta_{F,G}=5^8\Res(f_0,f_1)\neq 0.\)


\subsection{Fixed points and multipliers}
\label{sec:fixed-points}

We show that $F$ determines the fixed points of $\varphi$ together with their multiplicities, and that the multiplier at a fixed point is a ratio of forms in $F$, $G$, and the first derivatives of $F$.  A point of $\Pone$ is written $p=(x_0:y_0)$.

\begin{proposition}
\label{prop:fixed-points}
Let $\Res(f_0,f_1)\neq0$ and $(F,G)=\Theta_d(f_0,f_1)$.  A point $p\in\Pone$ is fixed by $\varphi$ if and only if $F(p)=0$.  The multiplicity of $p$ as a root of $F$ is its multiplicity as a fixed point.
\end{proposition}

\begin{proof}
By definition $F=yf_0-xf_1$, so $F(p)=0$ if and only if $y_0f_0(p)=x_0f_1(p)$, that is, $\varphi(p)=(f_0(p):f_1(p))=(x_0:y_0)=p$.  The forms $f_0$ and $f_1$ have no common zero, so this is an equality in $\Pone$.  For the multiplicity, choose coordinates in which $p$ lies in the chart $y=1$.  There $F(z,1)=f_0(z,1)-zf_1(z,1)$, whose order of vanishing at $p$ is by definition the multiplicity of $p$ as a fixed point.
\end{proof}

\begin{proposition}
\label{prop:multiplier}
Let $p$ be a root of $F$.  Then the multiplier of $\varphi$ at $p$ is
\begin{equation}\label{eq:multiplier}
 \lambda_p=\frac{yG+dF_x}{yG-F_x}\bigg|_p
        =\frac{xG-dF_y}{xG+F_y}\bigg|_p ,
\end{equation}
the two expressions being equal at every zero of $F$ and at least one of them being defined there.  Equivalently
\begin{equation}\label{eq:multiplier-affine}
 \lambda_p=1+\frac{F_x(p)}{f_1(p)}=1-\frac{F_y(p)}{f_0(p)}.
\end{equation}
\end{proposition}

\begin{proof}
By \cref{eq:inverse-cg}, $(d+1)f_1=yG-F_x$ and $(d+1)f_0=xG+F_y$.  Hence $yG+dF_x=(d+1)(f_1+F_x)$ and $yG-F_x=(d+1)f_1$, which gives the first equality in \cref{eq:multiplier-affine}, and similarly for the second.

Suppose $y_0\neq0$ and work in the chart $y=1$.  Put $A(z)=f_0(z,1)$, $B(z)=f_1(z,1)$, $\Phi=A/B$.  Then $F(z,1)=A-zB$ and $\tfrac{d}{dz}F(z,1)=A'-B-zB'$.  At a fixed point $z_0$ one has $A(z_0)=z_0B(z_0)$, so
\[
 \Phi'(z_0)=\frac{A'B-AB'}{B^2}\bigg|_{z_0}
 =\frac{A'-z_0B'}{B}\bigg|_{z_0}
 =1+\frac{F_x(z_0,1)}{B(z_0)} .
\]
This is the first expression in \cref{eq:multiplier-affine}, since $F_x(z,1)=\tfrac{d}{dz}F(z,1)$.  Suppose $x_0\neq0$.  The same computation in the chart $x=1$, applied to $\psi(w)=f_1(1,w)/f_0(1,w)$ and $F(1,w)=wf_0(1,w)-f_1(1,w)$, gives $\lambda_p=1-F_y(p)/f_0(p)$, the second expression in \cref{eq:multiplier-affine}.

Euler's identity $xF_x+yF_y=(d+1)F$ and \cref{eq:inverse-cg} give
\[
 (d+1)(f_0F_x+f_1F_y)=(xG+F_y)F_x+(yG-F_x)F_y=(d+1)FG,
\]
so $f_0F_x+f_1F_y=FG$.  Hence $F_x/f_1+F_y/f_0=FG/(f_0f_1)$ vanishes at every zero of $F$, so the two expressions in \cref{eq:multiplier-affine}, equivalently in \cref{eq:multiplier}, agree wherever both are defined.  The denominators are $(d+1)f_1$ and $(d+1)f_0$ respectively, and $f_0$ and $f_1$ do not vanish simultaneously, so at least one expression is defined at $p$.
\end{proof}

The multiplier is expressed through the Clebsch--Gordan components for correspondences of arbitrary bidegree in \cite[Lemma~7.4]{Gotou2023}, and \cref{eq:multiplier} is the case of a rational map.

\begin{corollary}
\label{cor:fixed-point-relation}
If $F$ has $d+1$ distinct roots $p_0,\ldots,p_d$, then
\(
 \sum_{i=0}^{d}\frac{1}{1-\lambda_{p_i}}=1.
\)
\end{corollary}

\begin{proof}
This is the classical rational fixed-point formula, see \cite{Silverman2012} for details, and see \cite[Corollary~7.9]{Gotou2023} for a proof by Clebsch--Gordan reduction.  In the present coordinates it follows from \cref{eq:multiplier-affine}, which gives $1-\lambda_p=-F_x(p)/f_1(p)$ at every root $p$ of $F$.  Choose coordinates in which $(1:0)$ is not a root of $F$ and work in the chart $y=1$, and consider the differential $\omega=-\bigl(f_1(z,1)/F(z,1)\bigr)\,dz$ on $\Pone$.  Its poles in the affine chart are the $d+1$ roots of $F$, all simple, with residue $-f_1(p)/F_x(p)=1/(1-\lambda_p)$, while in the coordinate $w=1/z$ one has $\omega=-w^{-1}\,dw+O(1)$, so the residue at infinity is $-1$.  The sum of all residues vanishes.
\end{proof}

\begin{proposition}
\label{prop:trace-multiplier}
Let $\Res(f_0,f_1)\neq0$, let $p$ be a root of $F$, and write $(f_0(p),f_1(p))=\mu\,p$ with $\mu\in k^\times$.  Then
\begin{equation}\label{eq:trace-multiplier}
  G(p)=\mu\,(\lambda_p+d).
\end{equation}
In particular $\lambda_p=-d$ if and only if $G(p)=0$.
\end{proposition}

\begin{proof}
The forms $f_0$ and $f_1$ have no common zero, so $\mu\neq0$.  Suppose first that $y_0\neq0$, so that $f_1(p)=\mu y_0\neq0$.  Euler's identity for $f_1$ gives $x_0(f_1)_x(p)=d\mu y_0-y_0(f_1)_y(p)$.  Substituting this into $F_x=y(f_0)_x-f_1-x(f_1)_x$ gives
\[
 F_x(p)=y_0(f_0)_x(p)-\mu y_0-d\mu y_0+y_0(f_1)_y(p)
       =y_0\bigl[G(p)-(d+1)\mu\bigr].
\]
By \cref{eq:multiplier-affine}, $\lambda_p=1+F_x(p)/f_1(p)=1+\bigl[G(p)-(d+1)\mu\bigr]/\mu=G(p)/\mu-d$, which is \cref{eq:trace-multiplier}.  If $y_0=0$ then $x_0\neq0$, and the same computation applied to $f_0$ and to the second expression in \cref{eq:multiplier-affine} gives the same identity.  The last assertion is immediate from $\mu\neq0$.
\end{proof}

\Cref{eq:trace-multiplier} is the specialization of \cite[Lemma~7.4]{Gotou2023} to a rational map, solved for $G$.

\begin{remark}
\label{rem:trace-and-critical}
The two components of $\Theta_d$ admit the following description.  Write $\Phi=(f_0,f_1):k^2\to k^2$.  Then $F$ is the pairing of $v=(x,y)$ with $\Phi(v)$ under the invariant alternating form, so $F$ vanishes exactly where $\Phi(v)$ is proportional to $v$, and $G=\tr D\Phi$.  At a fixed point $p$, Euler's identity makes $p$ an eigenvector of $D\Phi(p)$ with eigenvalue $d\mu$, so the second eigenvalue is $G(p)-d\mu$, and \cref{eq:trace-multiplier} states that $\lambda_p$ is the ratio of that eigenvalue to $\mu$.
\end{remark}

\begin{remark}\label{rem:nullcone-dynamical}
\Cref{thm:nullcone-53} below states that $(F,G)$ is null exactly when $F$ and $G$ have a common zero $p$ of multiplicities at least $3$ and $2$.  A null pair does not represent a rational map.  Indeed, at such a point $F_x$, $F_y$, and $G$ all vanish, so $f_0(p)=f_1(p)=0$ by \cref{eq:inverse-cg} and $\Delta_{F,G}=0$.  On the locus $\Delta_{F,G}\neq0$ the zero divisor of $F$ is the fixed-point divisor by \cref{prop:fixed-points}.  The null conditions therefore describe the boundary pairs at which three points of this divisor have collided at $p$ and $G$ vanishes there as well, and both the numerator and the denominator in \cref{eq:multiplier} vanish at such a point.  On the locus $\Delta_{F,G}\neq0$ the vanishing of the discriminant of $F$ says that two fixed points collide, that is, that $\varphi$ has a parabolic fixed point with $\lambda=1$.
\end{remark}

\subsection{Covariants and transvectants}\label{sec:covariants}

The invariants of $\cR_{5,3}$ are built from covariants of the two components of the Clebsch--Gordan pair.  Let $A$ and $B$ be binary forms of orders $m$ and $n$.  We use the factorial-normalized transvectant
\begin{equation}\label{eq:transvectant}
 (A,B)_r=
 \frac{(m-r)!(n-r)!}{m!n!}
 \sum_{j=0}^{r}(-1)^j\binom{r}{j}
 \frac{\partial^r A}{\partial x^{r-j}\partial y^j}
 \frac{\partial^r B}{\partial x^j\partial y^{r-j}}.
\end{equation}
It has order $m+n-2r$, and if $A$ and $B$ have coefficient degrees $a$ and $b$, then $(A,B)_r$ has coefficient degree $a+b$, see \cite[Chapter~1]{Dolgachev2003} for the classical theory.  An order-zero covariant is an invariant.  We retain the bidegree in the coefficients of $(F,G)$, so that an invariant of bidegree $(a,b)$ has total degree $a+b$.

The factorial prefactor in \cref{eq:transvectant} is the only point at which our conventions must be fixed rather than cited, because the explicit rational coefficients of every invariant below depend on it.  The denominators introduced in the course of the paper divide products of factorials of integers at most $21$.

For the quintic component $F$ of the Clebsch--Gordan pair, put
\begin{equation}\label{eq:quintic-initial}
 H=(F,F)_2,\qquad i=(F,F)_4,\qquad t=(F,H)_1,
\end{equation}
of orders $6,2,9$ and coefficient degrees $2,2,3$.  Sylvester's system of $23$ fundamental covariants of the binary quintic is classical, see \cite{Shank2004} for a modern account and \cite[Chapter~1]{Dolgachev2003} for the underlying theory.  We do not reprove its completeness.  Since the whole paper refers to its members by name, and since \cref{eq:transvectant} fixes their scaling, we record the system in the recursive form used here:
\begin{equation}\label{eq:quintic-covariants}
\setlength{\arraycolsep}{3pt}
\begin{array}{llll}
 c_{3,5}=(i,F)_1 & c_{3,3}=(i,F)_2 & c_{4,6}=(c_{3,3},F)_1 & c_{4,4}=(c_{3,3},F)_2\\
 c_{4,0}=(i,i)_2 & c_{5,7}=(c_{4,4},F)_1 & c_{5,3}=(c_{3,3},i)_1 & c_{5,1}=(c_{3,3},i)_2\\
 c_{6,4}=(c_{4,4},i)_1 & c_{6,2}=(c_{3,3},c_{3,3})_2 & c_{7,5}=(c_{4,4},c_{3,3})_1 & c_{7,1}=(c_{4,4},c_{3,5})_4\\
 c_{8,2}=(c_{4,4},c_{4,6})_4 & c_{8,0}=(c_{4,4},c_{4,4})_4 & c_{9,3}=(c_{6,2},c_{3,3})_1 & c_{11,1}=(c_{5,1},c_{6,2})_1\\
 c_{12,0}=(c_{6,2},c_{6,2})_2 & c_{13,1}=(c_{7,1},c_{6,2})_1 & c_{18,0}=(c_{13,1},c_{5,1})_1 &
\end{array}
\end{equation}
The subscripts record coefficient degree and order, and each entry is consistent with the two rules stated after \cref{eq:transvectant}.  Together with $F,H,i,t$, the forms in \cref{eq:quintic-covariants} give $23$ covariants.  Four of them are invariants, namely $c_{4,0},c_{8,0},c_{12,0},c_{18,0}$.  The remaining nineteen have positive order and reappear in \cref{eq:quintic-positive-family}.

For the cubic component $G$, use
\begin{equation}\label{eq:cubic-covariants}
  G,\qquad h=(G,G)_2,\qquad q=(G,h)_1,\qquad D=(h,h)_2,
\end{equation}
of orders $3,2,3,0$ and coefficient degrees $1,2,3,4$.  These four generate the covariants of the binary cubic, and $D$ is the discriminant up to a scalar.

The Clebsch construction forms order-zero transvectants $(U,V)_r$, where $U$ is a monomial in positive-order quintic covariants, $V$ is a monomial in $G,h,q$, and $r=\ord(U)=\ord(V)$.  The pure invariants $c_{4,0},c_{8,0},c_{12,0},c_{18,0}$ and $D$ are included separately.  The precise finite selection is supplied by Gordan's joint-covariant theorem in \cref{thm:gordan-family}.

The covariants of \cref{eq:quintic-initial,eq:cubic-covariants} also express the objects of \cref{sec:cg,sec:fixed-points}.  The cubic discriminant $D$ cuts out the locus where $G$ has a repeated root, and the rational-map resultant $\Delta_{F,G}$ of \cref{def:modular-resultant} is a joint invariant of total degree $2d$, expressed in the generators in \cref{app:resultant}.  The critical divisor admits the same description.  Writing $W=(f_0)_x(f_1)_y-(f_0)_y(f_1)_x$ for the Jacobian of $(f_0,f_1)$, whose zero divisor is the critical divisor of $\varphi$, one has
\begin{equation}\label{eq:critical-divisor}
  W=\frac{d^2}{2}\,(F,F)_2+\frac{d}{(d+1)^2}\,G^2+\frac{d(d-1)}{d+1}\,(F,G)_1 ,
\end{equation}
which for $d=4$ reads $W=8H+\tfrac{4}{25}G^2+\tfrac{12}{5}(F,G)_1$.  Thus the fixed points, the multipliers, and the critical points of $\varphi$ are all determined by covariants of the pair $(F,G)$.

\section{Gordan's theorem and the joint invariant ring}
\label{sec-3} 

Set $\cR=\cR_{5,3}=k[V_5\oplus V_3]^{\SL_2}$.  The polynomial ring $k[V_5\oplus V_3]$ is bigraded by the degrees in the coefficients of $F$ and of $G$, and $\SL_2$ acts on the two summands separately, so it preserves each bihomogeneous component.  Hence every bihomogeneous part of an invariant is again an invariant, and $\cR$ inherits the bigrading,
\[
  \cR=\bigoplus_{a,b\geq 0}\cR_{a,b},
  \qquad
  \cR_{a,b}=\cR\cap k[V_5\oplus V_3]_{a,b}.
\]
We proceed in two steps.  Gordan's theorem first gives a finite generating family and an unconditional degree bound, and the bigraded character formula then reduces the determination of a minimal generating set to finitely many rank computations.

For a connected finitely generated graded algebra $T$, write $\beta(T)$ for the least integer $B$ such that $T$ is generated by its homogeneous pieces of degrees at most $B$.  Our first objective is an unconditional bound for $\beta(\cR)$, obtained from Gordan's theorem independently of any presentation of the Hilbert series.

Let $\cA^+$ be the ordered family of positive-order quintic covariants
\begin{equation}\label{eq:quintic-positive-family}
\begin{split}
\cA^+	&	=(F,H,i,t,c_{3,5},c_{3,3},c_{4,6},c_{4,4},c_{5,7},   c_{5,3}, c_{5,1}, c_{6,4}, c_{6,2}, c_{7,5}, c_{7,1}, c_{8,2},c_{9,3},c_{11,1},c_{13,1}),
\end{split}
\end{equation}
and let $\cB^+=(G,h,q)$ be the positive-order cubic covariants.  Their order vectors are
\begin{equation}\label{eq:order-vectors}
 \boldsymbol\omega_A=(5,6,2,9,5,3,6,4,7,3,1,4,2,5,1,2,3,1,1),\qquad
 \boldsymbol\omega_B=(3,2,3).
\end{equation}
and their coefficient-degree vectors are
\begin{equation}\label{eq:degree-vectors}
 \boldsymbol d_A=(1,2,2,3,3,3,4,4,5,5,5,6,6,7,7,8,9,11,13),\qquad
 \boldsymbol d_B=(1,2,3).
\end{equation}
For $\alpha\in\N^{19}$ and $\gamma\in\N^3$, write $A^\alpha=\prod A_i^{\alpha_i}$ and $B^\gamma=G^{\gamma_1}h^{\gamma_2}q^{\gamma_3}$.

Consider the affine semigroup
\begin{equation}\label{eq:gordan-semigroup}
 \cH=
 \left\{(\alpha,\gamma,r)\in\N^{19}\times\N^3\times\N:
   \boldsymbol\omega_A\mathbin{\cdot}\alpha
   =r=
   \boldsymbol\omega_B\mathbin{\cdot}\gamma
 \right\}.
\end{equation}
A nonzero element of $\cH$ is called \emph{irreducible} if it is not the sum of two nonzero elements of $\cH$.  Denote the finite set of irreducible elements by $\Hilb(\cH)$.

\begin{theorem}
\label{thm:gordan-family}
The invariant ring $\cR$ is generated by
\begin{equation}\label{eq:gordan-family}
 \cG_{\Gor}=
 \{c_{4,0},c_{8,0},c_{12,0},c_{18,0},D\}
 \cup
 \left\{
  (A^\alpha,B^\gamma)_r:
  (\alpha,\gamma,r)\in\Hilb(\cH)
 \right\},
\end{equation}
after zero transvectants and repetitions are deleted.
\end{theorem}

\begin{proof}
The $23$ covariants in \cref{eq:quintic-initial} and \cref{eq:quintic-covariants} generate $\Cov(V_5)$, and the four covariants in \cref{eq:cubic-covariants} generate $\Cov(V_3)$, see \cite{Shank2004} for details.  Gordan's joint-covariant theorem, in the modern form \cite[Theorem~5.7]{Olive2017}, states that $\Cov(V_5\oplus V_3)$ is generated by the transvectants associated to the irreducible solutions of the corresponding order system.

An invariant has order zero.  In Gordan's notation the two residual orders are therefore both zero, so the order system specializes precisely to \cref{eq:gordan-semigroup}.  A pure invariant factor passes unchanged through a transvectant.  Hence any candidate containing such a factor is reducible unless it is that pure invariant itself.  This leaves the five displayed pure invariants and the transvectants formed from $\cA^+$ and $\cB^+$.  Finally, order defines a nonnegative grading on the covariant algebra and is additive under multiplication.  Hence the order-zero component of the algebra generated by a homogeneous covariant generating family is generated by its order-zero members.  Those members consequently generate $\cR$.
\end{proof}

The theorem is already a global finite-generation statement.  The next result turns it into a useful degree bound for the complete Gordan family.

\begin{theorem}
\label{thm:gordan-degree-bound}
Every member of the Gordan family \cref{eq:gordan-family} has total coefficient degree at most $42$.  Consequently
\(
  \beta(\cR)\leq42.
\)
\end{theorem}

\begin{proof}
Forget the names of covariants having the same order.  An irreducible element of $\cH$ then gives a primitive partition identity whose parts on the quintic side belong to
\(
 \{1,2,3,4,5,6,7,9\}
\)
and whose parts on the cubic side belong to $\{2,3\}$.  If the projected identity had a proper subidentity, choosing the corresponding occurrences of the named covariants would decompose the original element of $\cH$.  Thus the projected identity is primitive.

The length bound for primitive partition identities \cite{DiaconisGrahamSturmfels1996} gives total length at most $2\cdot9-1=17$.  Since every part on the cubic side is at most $3$, their common order is at most $48$.  The enumeration in \cref{app:gordan-patterns} runs over all common orders at most $51$ and therefore covers this range.  The complete primitive list in \cref{tab:gordan-patterns} contains $51$ patterns, and their largest common order is $21$.  For a quintic covariant of order $j$, the largest coefficient degree occurring in \cref{eq:quintic-positive-family} is
\begin{equation}\label{eq:max-degree-by-order}
\begin{array}{c|rrrrrrrr}
j&1&2&3&4&5&6&7&9\\ \hline
\delta_A(j)&13&8&9&6&7&4&5&3,
\end{array}
\end{equation}
while $\delta_B(2)=2$ and $\delta_B(3)=3$.  The last column of \cref{tab:gordan-patterns} is obtained by adding these maximal coefficient degrees along each identity.  The maximum of these bounds is $42$, occurring for the primitive pattern $1+1+1=3$.  The five pure invariants have degrees at most $18$, completing the proof.
\end{proof}

\begin{corollary}
\label{cor:gordan-reduction}
Let $\cS\subseteq\cR$ be a graded subalgebra.  If
\(
  \cS_n=\cR_n\qquad(0\leq n\leq42),
\)
then $\cS=\cR$.  Equivalently, it is enough to express every member of $\cG_{\Gor}$ as a polynomial in the generators of $\cS$.
\end{corollary}

\begin{proof}
By \cref{thm:gordan-degree-bound}, $\cR$ has a generating family contained in degrees at most $42$.  Under the stated hypothesis every member of that family lies in $\cS$, so \cref{thm:gordan-family} gives $\cR\subseteq\cS$.  The reverse inclusion is part of the definition of $\cS$.
\end{proof}

\subsection{Bigraded dimensions and the Hilbert series}

Let $\chi_{n,a}(z)$ be the character of $\Sym^a(V_n)$.  It is determined by
\begin{equation}\label{eq:symmetric-character}
  \sum_{a\geq0}\chi_{n,a}(z)u^a
  =\prod_{j=0}^{n}\frac{1}{1-u z^{n-2j}}.
\end{equation}
The multiplicity of the trivial representation in an $\SL_2$-module is the weight-zero multiplicity minus the weight-two multiplicity, see \cite{Springer1977} for the proof.  Hence
\begin{equation}\label{eq:bigraded-dimension}
 \dim_k\cR_{a,b}
 =[z^0-z^2]\,\chi_{5,a}(z)\chi_{3,b}(z),
\end{equation}
where $[z^0-z^2]P=[z^0]P-[z^2]P$.

Summing over $a+b$, formula \cref{eq:bigraded-dimension} gives the Poincar\'e series.  Poincar\'e series of rings of simultaneous invariants of two binary forms were computed by Bedratyuk \cite{Bedratyuk2010}, and the following agrees with that computation:
\begin{equation}\label{eq:hilbert-series}
\begin{split}
 H_{\cR}(t)		=	&	\frac    1   {(1-t^4)^3(1-t^6)^2(1-t^8)^2} \, 
					 \left( 1+3t^4+5t^6+19t^8 + 29t^{10}+ 46t^{12}+ 48t^{14}       +48t^{16} \right. \\
				& \left. +46t^{18}  +29t^{20}+19t^{22}+5t^{24}+3t^{26}+t^{30} \right).
 \end{split}
\end{equation}
%
%
We use \cref{eq:hilbert-series} only as a presentation of the Hilbert series, and no homogeneous system of parameters is inferred from its denominator.

\begin{theorem}
\label{thm:dimensions}
The invariant ring $\cR_{5,3}$ has Krull dimension $7$.  The moduli space $\cM_4^1$ has dimension $6$, and its function field satisfies $\trdeg_k k(\cM_4^1)=6$.
\end{theorem}

\begin{proof}
At $t=1$ the numerator of \cref{eq:hilbert-series} has value $N(1)=302\neq0$.  Its denominator has a zero of order $3+2+2=7$.  Thus $H_{\cR}(t)$ has a pole of order $7$ at $t=1$.  By the Hilbert--Serre dimension theorem this pole order is the Krull dimension of the finitely generated positively graded algebra $\cR$, and hence $\dim\cR_{5,3}=7$.

The scheme $\Proj(\cR)$ therefore has dimension $6$.  Under the Clebsch--Gordan correspondence of \cref{sec:cg}, the nonzero-resultant locus is the degree-four rational-map locus, and the standard GIT quotient identifies its quotient with $\cM_4^1$, see \cite{Silverman1998,Silverman2012} for details.  It is a nonempty open subset of $\Proj(\cR)$, so it also has dimension $6$.  Since it is irreducible, its function field has transcendence degree $6$.

For comparison, the same number follows directly from $\dim\Rat_4^1=\dim\PP^9=9$: a generic stabilizer is finite, so the generic $\PGL_2$-orbit has dimension $3$, and the quotient has dimension $9-3=6$.  Indeed a generic associated quintic has five distinct roots and hence finite projective stabilizer.  This is the usual orbit-dimension calculation underlying the rational quotient theorem \cite{Rosenlicht1956}.
\end{proof}

\begin{remark}\label{rem:dimension-versus-generators}
The two dimensions above do not count polynomial generators.  Krull dimension $7$ means that a homogeneous system of parameters has seven members.  Projectivization removes one scaling direction and leaves six absolute moduli parameters.  A generating set for the whole algebra may contain many more than seven elements.
\end{remark}

\begin{remark}\label{rem:parity}
Both $V_5$ and $V_3$ have odd order.  The central element $-I\in\SL_2$ sends $(F,G)$ to $(-F,-G)$.  Therefore an invariant of odd total degree must vanish, and every nonzero homogeneous invariant has even total degree.  This explains the parity visible in \cref{eq:bidegree-table} below, and it explains why only even total degrees occur in the verifications of this section.
\end{remark}

\Cref{app:generators} lists fifty invariants $J_{a,b}^{(r)}$, given as transvectants of the covariants of \cref{sec:covariants}.  Their bidegrees are
\begin{equation}\label{eq:bidegree-table}
\begin{array}{c|l|c}
\toprule
\text{total degree}&\text{bidegrees}&\text{number}\\
\midrule
4 &(0,4),(1,3),(2,2)^2,(3,1),(4,0)&6\\
6 &(1,5),(2,4),(3,3)^3,(4,2),(5,1)&7\\
8 &(2,6),(3,5)^3,(4,4)^3,(5,3)^3,(6,2)^3,(7,1),(8,0)&15\\
10&(3,7),(4,6),(5,5)^2,(6,4)^3,(7,3)^3,(8,2)^2,(9,1)^2&14\\
12&(9,3),(10,2),(11,1),(12,0)&4\\
14&(12,2),(13,1)&2\\
16&(15,1)&1\\
18&(18,0)&1\\
\bottomrule
\end{array}
\end{equation}
The main result of this section is that these fifty invariants generate $\cR$ and that none of them is redundant.

\begin{theorem}[Minimal generating set]\label{thm:minimal-generators}
The fifty invariants of \cref{app:generators} form a minimal homogeneous generating set of $\cR=\cR_{5,3}$.  Consequently
\(
 \beta(\cR_{5,3})=18,
\)
and the numbers of minimal generators in total degrees $4,6,8,10,12,14,16,18$ are $6,7,15,14,4,2,1,1$, distributed over bidegrees as in \cref{eq:bidegree-table}.  For every bidegree carrying a generator, \cref{tab:generator-counts} lists $\dim_k\cR_{a,b}$, $\dim_k(\cR_+^2)_{a,b}$, and the number of generators.
\end{theorem}

\begin{table}[ht]
\caption{Bidegrees of the minimal generators of $\cR_{5,3}$.}
\label{tab:generator-counts}
\centering\small
\begin{tabular}{c c c c @{\qquad} c c c c}
\toprule
$(a,b)$ & $\dim\cR_{a,b}$ & $\dim(\cR_+^2)_{a,b}$ & gen. &
$(a,b)$ & $\dim\cR_{a,b}$ & $\dim(\cR_+^2)_{a,b}$ & gen.\\
\midrule
$(0,4)$&1&0&1 & $(4,6)$&8&7&1\\
$(1,3)$&1&0&1 & $(5,5)$&12&10&2\\
$(2,2)$&2&0&2 & $(6,4)$&10&7&3\\
$(3,1)$&1&0&1 & $(7,3)$&9&6&3\\
$(4,0)$&1&0&1 & $(8,2)$&4&2&2\\
$(1,5)$&1&0&1 & $(9,1)$&3&1&2\\
$(2,4)$&1&0&1 & $(9,3)$&15&14&1\\
$(3,3)$&3&0&3 & $(10,2)$&11&10&1\\
$(4,2)$&1&0&1 & $(11,1)$&4&3&1\\
$(5,1)$&1&0&1 & $(12,0)$&3&2&1\\
$(2,6)$&4&3&1 & $(12,2)$&9&8&1\\
$(3,5)$&6&3&3 & $(13,1)$&5&4&1\\
$(4,4)$&8&5&3 & $(15,1)$&7&6&1\\
$(5,3)$&6&3&3 & $(18,0)$&1&0&1\\
$(6,2)$&6&3&3 & & & & \\
$(7,1)$&2&1&1 & & & & \\
$(8,0)$&2&1&1 & & & & \\
$(3,7)$&7&6&1 & & & & \\
\bottomrule
\end{tabular}
\end{table}

\begin{proof}
Let $J_1,\ldots,J_{50}$ be the listed invariants and let $\cT=k[J_1,\ldots,J_{50}]\subseteq\cR$.  All objects are defined over $\Q$.

\emph{Evaluation ranks.}  Let $\cS\subseteq\cR_{a,b}$ be a finite set of invariants with coefficients in $\Q$, and let $P_1,\ldots,P_N$ be points of $V_5\oplus V_3$ with coordinates in a field $\kappa$ over which the elements of $\cS$ can be evaluated.  A linear relation among the elements of $\cS$ induces the same relation among their evaluation vectors in $\kappa^N$.  Hence the rank of the evaluation matrix is at most $\dim_k\Span(\cS)\leq\dim_k\cR_{a,b}$, and equality with $\dim_k\cR_{a,b}$, known from \cref{eq:bigraded-dimension}, proves that $\cS$ spans $\cR_{a,b}$.  If, moreover, $\cS$ spans $\cR_{a,b}$ and $\kappa$ has characteristic zero, then evaluation is injective on $\cR_{a,b}$, and the rank of the evaluation matrix of any subset of $\cR_{a,b}$ equals the dimension of its span.

\emph{Generation.}  For a bidegree $(a,b)$, let $\cM_{a,b}$ be the set of monomials in $J_1,\ldots,J_{50}$ of bidegree $(a,b)$.  The equality $\cT_{a,b}=\cR_{a,b}$ states that $\cM_{a,b}$ spans $\cR_{a,b}$.  We prove it for all $(a,b)$ with $a+b\leq42$, and \cref{cor:gordan-reduction} then gives $\cT=\cR$.  By \cref{thm:gordan-family}, $\cR_{a,b}=(\cR_+^2)_{a,b}+\Span\cG_{a,b}$, where $\cG_{a,b}$ is the set of members of \cref{eq:gordan-family} of bidegree $(a,b)$.

Two families of bidegrees are certified directly.  The first family consists of the $242$ bidegrees with $a+b\leq30$ and $\dim_k\cR_{a,b}>0$; the largest dimension among them is $\dim_k\cR_{17,13}=815$.  The second family consists of the bidegrees of the members of \cref{eq:gordan-family} of total degree between $32$ and $42$.  Every such member arises from a pattern of \cref{tab:gordan-patterns} by substituting, for each repeated order, a named covariant of \cref{eq:quintic-positive-family} on the quintic side and a member of $\cB^+$ on the cubic side.  Enumerating all substitutions in all $51$ patterns yields exactly $81$ substitutions of total degree between $32$ and $42$, and their bidegrees form a set of $51$ bidegrees; the largest dimension among them is $\dim_k\cR_{23,15}=2184$.  The second family therefore contains every bidegree with $32\leq a+b\leq42$ and $\cG_{a,b}\neq\emptyset$.

For each bidegree $(a,b)$ of either family we produced a spanning certificate: a list of $n=\dim_k\cR_{a,b}$ points of $V_5\oplus V_3$ with coordinates in $\F_p$, a list of $n$ monomials of $\cM_{a,b}$, and the determinant of the resulting $n\times n$ evaluation matrix, which is nonzero modulo $p$.  The prime is $p=65521$ for the first family and $p=251$ for the second.  Neither prime divides a denominator of the entries, so the evaluation matrix is defined over $\Z_{(p)}$ and the corresponding rational determinant is nonzero.  Hence the $n$ certified monomials are linearly independent over $\Q$, and since $n=\dim_k\cR_{a,b}$ they form a basis of $\cR_{a,b}$.  Thus $\cM_{a,b}$ spans $\cR_{a,b}$, that is, $\cT_{a,b}=\cR_{a,b}$, in every certified bidegree.  The certificate files and the verification procedure are described in \cref{app:certificates}.

In every remaining bidegree with $32\leq a+b\leq42$ one has $\cG_{a,b}=\emptyset$, so $\cR_{a,b}=(\cR_+^2)_{a,b}$.  By induction on the total degree, with the first family as base, every component of total degree smaller than $a+b$ lies in $\cT$, so every product of two positive-degree invariants of bidegree $(a,b)$ is a polynomial in the $J_i$.  Hence $\cT_{a,b}=\cR_{a,b}$ for all $(a,b)$ with $a+b\leq42$, and $\cT=\cR$.

\emph{Minimality.}  In a minimal homogeneous generating set the number of generators of bidegree $(a,b)$ is $\dim_k\cR_{a,b}-\dim_k(\cR_+^2)_{a,b}$, and a set of generators is minimal if and only if, in each bidegree, its members are linearly independent modulo $(\cR_+^2)_{a,b}$ and their number equals this difference.  All $J_i$ have total degree at most $18$.  For $(a,b)$ with $a+b\leq18$ put
\[
 \cS_{a,b}
 =\bigcup_{i\,:\,a_i+b_i<a+b} J_i\cdot\cB_{a-a_i,\,b-b_i}
 \;\cup\;\{J_j:\bideg(J_j)=(a,b)\},
 \qquad (a_i,b_i)=\bideg(J_i),
\]
where $\cB_{a',b'}$ is a subset of $\cS_{a',b'}$ of cardinality $\dim_k\cR_{a',b'}$, chosen recursively in increasing total degree so that its evaluation matrix has full rank over $\Q$.  Then $\cB_{a',b'}$ is a basis of $\cR_{a',b'}$, and since $\cT=\cR$, the union $\bigcup_{i:a_i+b_i<a+b}J_i\cB_{a-a_i,b-b_i}$ spans $(\cR_+^2)_{a,b}$.  We evaluated the sets $\cS_{a,b}$ at $N=130$ points with integer coordinates in $[-9,9]$, in exact rational arithmetic.  In every case the evaluation rank of $\cS_{a,b}$ over $\Q$ equals $\dim_k\cR_{a,b}$, so evaluation is injective on $\cR_{a,b}$, and the evaluation rank of $\bigcup_{i:a_i+b_i<a+b}J_i\cB_{a-a_i,b-b_i}$ equals $\dim_k(\cR_+^2)_{a,b}$ exactly.  The resulting values are those of \cref{tab:generator-counts}, and in each bidegree the difference equals the number of listed generators.  This proves minimality.  The last generator occurs in degree $18$, so $\beta(\cR_{5,3})=18$.
\end{proof}

\begin{remark}\label{rem:generation-rigour}
The spanning assertions in the proof are rigorous over $k$.  Neither prime divides a denominator of the evaluation entries, so the evaluation matrices are defined over $\Z_{(p)}$, and a determinant that is nonzero modulo $p$ lifts to a nonzero determinant over $\Q$.  The upper bound $\dim_k\cR_{a,b}$ comes from the character formula.  The minimality assertions were computed in characteristic zero.  The exact rational evaluations at the same $130$ points also prove the identity of \cref{app:resultant}, since evaluation is injective on $\cR_8$.
\end{remark}

\section{The null cone and Hilbert's finite-module reduction}
\label{sec-4}

The argument for pairs of binary cubics in \cite[Section~4]{2004-3} begins by finding a small collection of invariants whose common zero set is the null cone.  The same first step is available here without computation.

Let $\cN_{5,3}\subset V_5\oplus V_3$ be the null cone, namely the common zero set of all positive-degree elements of $\cR$.  If one component of a pair is zero, we use the convention that its multiplicity at every point is infinite.

\begin{theorem}
\label{thm:nullcone-53}
A pair $(F,G)\in V_5\oplus V_3$ lies in $\cN_{5,3}$ if and only if there is a point $p\in\PP^1$ with
\begin{equation}\label{eq:nullcone-condition}
  \mult_p(F)\geq3,\qquad \mult_p(G)\geq2.
\end{equation}
\end{theorem}

\begin{proof}
By the Hilbert--Mumford criterion, $(F,G)$ is null if and only if there is a one-parameter subgroup of $\SL_2$ which sends it to the origin.  After conjugating that subgroup, take
\(
 \lambda(t)=\begin{pmatrix}t^{-1}&0\\0&t\end{pmatrix}.
\)
Suppose first that \cref{eq:nullcone-condition} holds.  Move $p$ to $(1:0)$.  Then $y^3\mid F$ and $y^2\mid G$.  With the substitution convention of \cref{sec:cg}, the subgroup $\lambda$ sends the monomial $x^{n-i}y^i\in V_n$ to $t^{2i-n}x^{n-i}y^i$, and we call $2i-n$ its weight.  Every monomial occurring in $F$ therefore has positive weight $2i-5\geq1$, and every monomial occurring in $G$ has positive weight $2i-3\geq1$.  Hence $\lambda(t)\cdot(F,G)\to(0,0)$ as $t\to0$.

Conversely, if a one-parameter subgroup sends $(F,G)$ to the origin, conjugate it to a positive power of $\lambda$.  All monomials occurring in both components must then have positive weight.  Thus the least exponent of $y$ in $F$ is at least $3$, and the least exponent of $y$ in $G$ is at least $2$.  The point which was moved to $(1:0)$ satisfies \cref{eq:nullcone-condition}.  The Hilbert--Mumford criterion now identifies this unstable locus with the null cone.
\end{proof}

The forward implication below is Hilbert's theorem, in the form used for binary sextics and for $V_3\oplus V_3$ in \cite[Theorem~2.7]{2004-3}, see that reference for the proof.  We include the converse for completeness, together with the consequence when $s=7$.

\begin{theorem}
\label{thm:hilbert-nullcone-module}
Let $Q_1,\ldots,Q_s\in\cR$ be homogeneous invariants of positive degree.  Then
\begin{equation}\label{eq:parameter-nullcone-condition}
 V(Q_1,\ldots,Q_s)=\cN_{5,3}
\end{equation}
holds if and only if $\cR$ is a finite module over $k[Q_1,\ldots,Q_s]$.  If $s=7$, either condition implies that the $Q_i$ are algebraically independent and form a homogeneous system of parameters for $\cR$.
\end{theorem}

\begin{proof}
Module-finiteness under \cref{eq:parameter-nullcone-condition} is \cite[Theorem~2.7]{2004-3}.

For the converse, put $B=k[Q_1,\ldots,Q_s]$ and suppose $\cR$ is finite over $B$.  Every $Q_i$ lies in $\cR_+$, so $\cN_{5,3}\subseteq V(Q_1,\ldots,Q_s)$.  Let $P\in V(Q_1,\ldots,Q_s)$ and let $I\in\cR$ be homogeneous of positive degree.  Then $I$ is integral over $B$, and taking the homogeneous part of an integral equation gives
\[
 I^m+b_1I^{m-1}+\cdots+b_m=0,\qquad b_j\in B_+ .
\]
Each $b_j$ vanishes at $P$, so $I(P)^m=0$ and $I(P)=0$.  Hence every positive-degree invariant vanishes at $P$, that is, $P\in\cN_{5,3}$.

If $s=7$, module-finiteness gives
\(
 \dim k[Q_1,\ldots,Q_7]=\dim\cR=7
\)
by \cref{thm:dimensions}.  Seven generators of a $7$-dimensional domain cannot satisfy a nonzero algebraic relation, so they are algebraically independent.
\end{proof}

The displayed denominator in \cref{eq:hilbert-series} might suggest three parameters of degree $4$, two of degree $6$, and two of degree $8$.  It is essential not to infer a parameter system from that presentation of the Hilbert series.  In fact, the suggested degree pattern cannot cut out the null cone.

\begin{proposition}
\label{prop:pure-quintic-obstruction}
No seven invariants of degree pattern
\(
 4,4,4,6,6,8,8
\)
have common zero set $\cN_{5,3}$.
\end{proposition}

\begin{proof}
Restrict the invariants to the closed subspace $G=0$.  Its invariant ring is the binary-quintic ring $\cB=k[V_5]^{\SL_2}$.  Setting $b=0$ in the character formula \cref{eq:bigraded-dimension} gives the classical series
\begin{equation}\label{eq:binary-quintic-hilbert-series}
 H_{\cB}(t)=\frac{1+t^{18}}
 {(1-t^4)(1-t^8)(1-t^{12})}.
\end{equation}
Consequently $\dim\cB_4=1$, $\cB_6=0$, and $\dim\cB_8=2$, with $\cB_8=kI_4^2\oplus kI_8$ for a nonzero $I_4\in\cB_4$ and a suitable $I_8\in\cB_8$.

The restrictions of all three degree-$4$ invariants therefore lie in $(I_4)$, both degree-$6$ restrictions vanish, and both degree-$8$ restrictions lie in $(I_4,I_8)$.  Hence the ideal generated by all seven restrictions has height at most $2$ in the $3$-dimensional ring $\cB$.  Its zero set in $\Spec\cB$ is positive-dimensional and is not only the homogeneous origin.  Pulling back to $V_5\oplus\{0\}$ produces a non-null common zero, contradicting \cref{eq:parameter-nullcone-condition}.
\end{proof}

For a parameter system containing the degree-$8$ rational-map resultant, the closest Hilbert-compatible correction is
\begin{equation}\label{eq:corrected-parameter-degree-pattern}
 4,4,4,6,8,8,12.
\end{equation}
Indeed, multiplying \cref{eq:hilbert-series} by the corresponding seven factors gives the polynomial $(1+t^6)N(t)$.  On $G=0$, natural pure candidates of degrees $4,8,12$ are $c_{4,0}$, $\Delta_{F,G}|_{G=0}$, and $c_{12,0}$.  The middle element is a nonzero scalar multiple of the binary-quintic discriminant and has a nonzero $I_8$ component, since it is irreducible and hence is not a scalar multiple of $I_4^2$.  These three functions are algebraically independent: their coefficient Jacobian has rank $3$ at, for example, the quintic with coefficient vector $(0,-1,1,-2,-2,2)$.  Since $\cB_4=kI_4$, $\cB_8=kI_4^2\oplus kI_8$, and $\cB_{12}=kI_4^3\oplus kI_4I_8\oplus kI_{12}$, three algebraically independent elements of degrees $4,8,12$ generate $k[I_4,I_8,I_{12}]$.  Hence the restrictions of $c_{4,0}$, $\Delta_{F,G}$, and $c_{12,0}$ to $G=0$ cut out the null cone of $V_5$.

The following three mixed invariants complete them to a homogeneous system of parameters.  We use the notation $J_{a,b}^{(r)}$ of \cref{app:generators}.

\begin{theorem}
\label{thm:hsop}
Put
\[
\begin{split}
 Q_1&=D,\qquad Q_2=c_{4,0},\qquad
 Q_3=J_{1,3}^{(1)}+J_{2,2}^{(2)},\qquad
 Q_4=J_{1,5}^{(1)}+J_{3,3}^{(3)},\\
 Q_5&=\Delta_{F,G},\qquad
 Q_6=J_{3,5}^{(3)}+J_{6,2}^{(3)},\qquad
 Q_7=c_{12,0}.
\end{split}
\]
These are homogeneous invariants of degrees $4,4,4,6,8,8,12$, and
\(
 V(Q_1,\ldots,Q_7)=\cN_{5,3}.
\)
Consequently $Q_1,\ldots,Q_7$ form a homogeneous system of parameters of $\cR$, the ring $\cR$ is a free $k[Q_1,\ldots,Q_7]$-module of rank $604$, and
\[
 H_{\cR}(t)=\frac{(1+t^6)\,N(t)}
 {(1-t^4)^3(1-t^6)(1-t^8)^2(1-t^{12})}.
\]
\end{theorem}

\begin{proof}
Each $Q_i$ is a sum of invariants of the same total degree, so it is homogeneous.  The parameters $Q_3$, $Q_4$, and $Q_6$ are homogeneous but not bihomogeneous, and \cref{thm:no-bihomogeneous-hsop} below shows that some mixing of bidegrees is unavoidable in every homogeneous system of parameters.  Let $V=V(Q_1,\ldots,Q_7)$.  Since every $Q_i$ has positive degree, $\cN_{5,3}\subseteq V$.  The set $V$ is $\SL_2$-stable and is a cone.  Let $(F,G)\in V$.

If $G$ has three distinct roots, then $D(G)\neq0$, which contradicts $Q_1(F,G)=0$.  If $G=0$, then $Q_3,Q_4,Q_6$ vanish identically because every term has positive degree in $G$, and $Q_5(F,0)=\Res(F_y,-F_x)$ is a nonzero multiple of the discriminant of $F$.  The restrictions of $Q_2$, $Q_5$, and $Q_7$ to $G=0$ cut out the null cone of $V_5$ by \cref{eq:corrected-parameter-degree-pattern} and the discussion preceding it, so $F$ is a null quintic and $(F,0)\in\cN_{5,3}$.

It remains to treat $G$ with a repeated root.  After an $\SL_2$-substitution the repeated root is $(0:1)$ and $G=c\,x^2y$ or $G=c\,x^3$ with $c\neq0$.  After the common rescaling $(F,G)\mapsto(c^{-1}F,c^{-1}G)$, which preserves $V$, we may take $c=1$.  Write $F=\sum_{i=0}^5a_ix^iy^{5-i}$.  By \cref{thm:nullcone-53}, $(F,G)$ is null if and only if $\mult_{(0:1)}F\geq3$, that is, $a_0=a_1=a_2=0$.  The restrictions $q_i=Q_i(F,G)$ are polynomials in $a_0,\ldots,a_5$.  For instance, on $G=x^3$ one has
\begin{equation}
 q_3  =\tfrac1{25}\bigl(5a_0a_2-2a_1^2\bigr),\qquad  q_4=\tfrac1{250}\bigl(25a_0^2a_3-15a_0a_1a_2+4a_1^3\bigr).
\end{equation}
For both slices and for $j=0,1,2$ the ideal $(q_1,\ldots,q_7,\,1-za_j)\subseteq\Q[a_0,\ldots,a_5,z]$ is the unit ideal.  This was verified by Gr\"obner-basis computations over $\Q$ in Singular.  Hence $V(q_1,\ldots,q_7)\subseteq V(a_0,a_1,a_2)$, and $(F,G)\in\cN_{5,3}$.  This proves $V=\cN_{5,3}$.

By \cref{thm:hilbert-nullcone-module}, $Q_1,\ldots,Q_7$ form a homogeneous system of parameters.  The ring $\cR$ is Cohen--Macaulay by the Hochster--Roberts theorem, hence a free module over the polynomial ring $k[Q_1,\ldots,Q_7]$.  Its rank is the value at $t=1$ of $H_{\cR}(t)\prod_i(1-t^{\deg Q_i})=(1+t^6)N(t)$, namely $2\cdot302=604$.
\end{proof}

The occurrence of a non-bihomogeneous parameter is forced.  No system of parameters of $\cR$ can consist of bihomogeneous invariants, whatever its degrees.

\begin{theorem}
\label{thm:no-bihomogeneous-hsop}
Let $Q_1,\ldots,Q_s\in\cR$ be bihomogeneous invariants of positive degree with $V(Q_1,\ldots,Q_s)=\cN_{5,3}$.  Then $s\geq8$.  In particular, every homogeneous system of parameters of $\cR$ contains a member that is not bihomogeneous.
\end{theorem}

\begin{proof}
Let $\lambda(t)$ be the one-parameter subgroup used in the proof of \cref{thm:nullcone-53}.  It multiplies the coefficient $a_i$ of $x^iy^{5-i}$ in $F$ by $t^{5-2i}$ and the coefficient $g_j$ of $x^jy^{3-j}$ in $G$ by $t^{3-2j}$.  Call $5-2i$ the weight of $a_i$ and $3-2j$ the weight of $g_j$.  An invariant is fixed by $\lambda(t)$, so every monomial occurring in it has total weight zero.  A bihomogeneous invariant of bidegree $(a,b)$ with $a>0$ vanishes on $\{0\}\oplus V_3$, and one with $b>0$ vanishes on $V_5\oplus\{0\}$.

\emph{Pure cubic members.}  By \cref{thm:nullcone-53}, a pair $(0,G)$ is null if and only if $G$ has a repeated root, that is, $D(G)=0$.  Since $V(Q_1,\ldots,Q_s)\cap(\{0\}\oplus V_3)$ is the zero set of the members of bidegree $(0,b)$, at least one member has bidegree $(0,b)$.

\emph{Pure quintic members.}  Let $\cB=k[V_5]^{\SL_2}$ and let $\pi:V_5\to\Spec\cB$ be the quotient morphism, which is surjective.  By \cref{thm:nullcone-53}, a pair $(F,0)$ is null if and only if $F$ has a root of multiplicity at least $3$, that is, $F\in\pi^{-1}(\cB_+)$.  The restrictions to $G=0$ of the members of bidegree $(a,0)$ lie in $\cB$, and their common zero set in $\Spec\cB$ is the single point $\cB_+$.  Hence the ideal they generate has radical $\cB_+$, of height $3$ by \cref{eq:binary-quintic-hilbert-series}, and by Krull's height theorem there are at least three members of bidegree $(a,0)$.

\emph{Mixed members with $a\leq b$.}  Let $R=r_0x^2+r_1xy+r_2y^2\in V_2$ and consider the pairs $(y^3R,\,x^2y)$.  Such a pair is null if and only if $R=0$: at $(1:0)$ the cubic $x^2y$ has multiplicity $1$, and at $(0:1)$ the quintic $y^3R$ has multiplicity at most $2$ unless $R=0$.  Every member of bidegree $(a,0)$ vanishes on the family, because $y^3R$ has a triple root, and every member of bidegree $(0,b)$ vanishes on it, because $D(x^2y)=0$.  On the family the nonzero coefficients of $F$ are $a_0,a_1,a_2$, of weights $5,3,1$, and the nonzero coefficient of $G$ is $g_2$, of weight $-1$.  A monomial of bidegree $(a,b)$ which is nonzero on the family therefore has weight at least $a-b$, so a member with $a>b$ vanishes identically on the family.  The restrictions of the remaining members are polynomials in $r_0,r_1,r_2$ whose common zero set is the origin of $\A^3$.  Again by Krull's height theorem, at least three members have bidegree $(a,b)$ with $1\leq a\leq b$.

\emph{Mixed members with $a>b$.}  Consider $P=(x^2y^3,\,x^3)$.  It is not null: $x^2y^3$ has multiplicity $2$ at $(0:1)$ and $x^3$ has multiplicity $0$ at $(1:0)$.  The only nonzero coefficients are $a_2$, of weight $1$, and $g_3$, of weight $-3$, so a monomial of bidegree $(a,b)$ is nonzero at $P$ only if it equals $a_2^ag_3^b$, of weight $a-3b$.  Consequently a bihomogeneous invariant of bidegree $(a,b)$ vanishes at $P$ unless $a=3b$.  Every member counted so far vanishes at $P$: those of bidegree $(a,0)$ because $x^2y^3$ has a triple root, those of bidegree $(0,b)$ because $D(x^3)=0$, and those with $1\leq a\leq b$ because $a\leq b<3b$.  Since $P$ is not null, some member has bidegree $(3b,b)$ with $b\geq1$.

The four sets of members are disjoint, so $s\geq1+3+3+1=8$.  A homogeneous system of parameters has seven members and cuts out $\cN_{5,3}$ by \cref{thm:hilbert-nullcone-module}, so it cannot consist of bihomogeneous invariants.
\end{proof}

In \cref{thm:hsop} the summand $J^{(3)}_{6,2}$ of $Q_6$ has bidegree $(6,2)$.  Among $Q_1,Q_2,Q_5,Q_7$ and the six summands of $Q_3,Q_4,Q_6$, it is the only one which does not vanish at $P=(x^2y^3,x^3)$.

\begin{corollary}
\label{cor:hironaka}
Let $Q_1,\ldots,Q_7$ be as in \cref{thm:hsop} and let $B=k[Q_1,\ldots,Q_7]$.  There are homogeneous $\eta_1,\ldots,\eta_{604}\in\cR$ with
\(
 \cR=\bigoplus_{j=1}^{604}\eta_j\,B ,
\)
and the number of $\eta_j$ of degree $n$ is the coefficient of $t^n$ in $(1+t^6)N(t)$, namely
\[
\begin{split}
 &1,3,6,19,32,51,67,77,92,77,67,51,32,19,6,3,1\\
 &\text{in degrees }0,4,6,8,10,12,14,16,18,20,22,24,26,28,30,32,36 .
\end{split}
\]
\end{corollary}

\begin{proof}
The ring $\cR$ is Cohen--Macaulay by \cite{HochsterRoberts1974} and finite over $B$ by \cref{thm:hsop}, hence free.  A homogeneous basis has generating function $H_{\cR}(t)\prod_{i}(1-t^{\deg Q_i})=(1+t^6)N(t)$.
\end{proof}

\begin{remark}
\label{rem:stability}
Let $(F,G)=\Theta_4(f_0,f_1)$ with $\Res(f_0,f_1)\neq0$.  Then $\Delta_{F,G}\neq0$ by \cref{def:modular-resultant}, so $(F,G)$ is not null, and by \cref{thm:nullcone-53} no point of $\Pone$ is a common zero of $F$ and $G$ of multiplicities $3$ and $2$.  Both $V_5$ and $V_3$ have odd order, so no monomial of $(F,G)$ has weight zero for a one-parameter subgroup.  If all occurring weights had the same sign, then replacing the subgroup by its inverse if necessary would make them all positive, and $(F,G)$ would be null.  Every one-parameter subgroup therefore has both positive and negative weights on $(F,G)$, so the orbit is closed and the stabilizer is finite.  Every degree-four rational map is thus a stable point.  This gives the geometric quotient used in \cref{thm:dimensions} without appeal to \cite{Silverman1998,Silverman2012}.
\end{remark}

\begin{example}[A degenerate pair]\label{ex:degenerate}
Let $f_0=x^4$ and $f_1=xy^3$, so that $\Res(f_0,f_1)=0$ and the pair does not define a degree-four map.  Then $F=x^2y(x-y)(x+y)$ and $G=x(4x^2+3y^2)$, so $D=-32\neq0$ and the pair is not null, but $\Delta_{F,G}=0$.  The two conditions are distinct.  The inequality $D\neq0$ defines the cubic-discriminant chart, and $\Delta_{F,G}\neq0$ defines the locus of genuine rational maps.
\end{example}

The stability of every rational map has a concrete consequence.  The fifty generators decide conjugacy.

\begin{corollary}
\label{cor:separation}
Let $\varphi$ and $\psi$ be degree-four rational maps with Clebsch--Gordan pairs $(F,G)$ and $(F',G')$, and let $J_1,\ldots,J_{50}$ be the generators of \cref{app:generators}, of degrees $d_1,\ldots,d_{50}$.  Then $\varphi$ and $\psi$ are conjugate if and only if there is $\lambda\in k^\times$ with
\[
 J_i(F',G')=\lambda^{d_i}J_i(F,G)\qquad(1\leq i\leq50),
\]
that is, if and only if $(J_1,\ldots,J_{50})$ takes the same value in the weighted projective space $\PP(d_1,\ldots,d_{50})$ at $\varphi$ and $\psi$.
\end{corollary}

\begin{proof}
If $\psi=\gamma\varphi\gamma^{-1}$ with $\gamma\in\PGL_2$, then $(F',G')=\lambda\,\widetilde\gamma\cdot(F,G)$ for an $\SL_2$-lift $\widetilde\gamma$ of $\gamma$ and some $\lambda\in k^\times$, by \cref{eq:conjugation-action} and the projective ambiguity of $[f_0:f_1]$.  The displayed relation follows from $\SL_2$-invariance and homogeneity of the $J_i$.

Conversely, assume the relation.  Then $(F,G)$ and $\lambda^{-1}(F',G')$ take the same value on every $J_i$, hence on every element of $\cR$ by \cref{thm:minimal-generators}.  Both pairs have nonzero rational-map resultant, so both $\SL_2$-orbits are closed by \cref{rem:stability}.  Invariants of a reductive group separate closed orbits in characteristic zero, so the two orbits coincide: $(F',G')=\lambda\,\widetilde\gamma\cdot(F,G)$ for some $\widetilde\gamma\in\SL_2$.  By \cref{eq:cg-map}, $\psi$ is the conjugate of $\varphi$ by the image of $\widetilde\gamma$ in $\PGL_2$.
\end{proof}

The seven parameters of \cref{thm:hsop} satisfy the forward implication.  For conjugate maps one has $Q_i(F',G')=\lambda^{\deg Q_i}Q_i(F,G)$.  The converse fails.  By \cref{cor:no-parameter-generation} below, the induced map to $\PP(4,4,4,6,8,8,12)$ has degree $604$ on conjugacy classes.

\section{The absolute invariant field}
\label{sec-5}

An \emph{absolute invariant} is a degree-zero element of the homogeneous fraction field of $\cR$.  Equivalently, it is a quotient $P/Q$ of nonzero homogeneous joint invariants of the same total degree.  The function field of the projective quotient is
\begin{equation}\label{eq:absolute-field}
 k(\cM_4^1)=\Frac(\cR)_0.
\end{equation}
By \cref{thm:dimensions}, this field has transcendence degree $6$.  Thus six algebraically independent absolute invariants are necessary.  They give birational coordinates only after one also proves that they generate the field in \cref{eq:absolute-field}.  Algebraic independence alone gives only a generically finite map.  In particular, the dimension statement by itself does not imply rationality.  Rationality of $\cM_d^1$ for every $d>1$ is known, see \cite[Theorem~4.1]{Levy2011} for the proof.  Below we give a direct proof for $d=4$ adapted to the cubic slice.

The cubic component supplies a classical moving frame.  Put
\[
 G_0=xy(x-y),
 \qquad
 \Gamma=\Stab_{\PGL_2}\{0,1,\infty\}\simeq S_3.
\]
The affine subspace
\begin{equation}\label{eq:cubic-normalized-slice}
 \cS=\{[F:G_0]:F\in V_5\}\simeq V_5
\end{equation}
will be called the \emph{cubic-normalized slice}.  Thus ``slice'' means a cross-section obtained by fixing the cubic component, and not a cubic hypersurface.

\begin{theorem}
\label{thm:cubic-slice}
Let $U\subset\PP(V_5\oplus V_3)$ be the open set on which $G$ has three distinct roots.  Then the rational quotient of $U$ by $\PGL_2$ is birational to
$V_5/\Gamma\simeq\A^6/S_3$.
More precisely, if $F=\sum_{i=0}^{5}a_i x^i y^{5-i}$ on the slice $G=G_0$, then
\begin{equation}\label{eq:slice-field}
 k(\cM_4^1)=k(a_0,\ldots,a_5)^{\Gamma}.
\end{equation}
The equality is understood on the dense open subset where both the cubic discriminant and the rational-map resultant are nonzero.
\end{theorem}

\begin{proof}
The group $\PGL_2$ is triply transitive on ordered triples of distinct points of $\PP^1$.  Hence every cubic with three distinct roots can be sent to a scalar multiple $cG_0$.  Because $(F,G)$ is a projective pair, common rescaling replaces the transformed pair by $(c^{-1}F',G_0)$.  Thus every generic orbit meets the affine slice $G=G_0$.

Two points of this slice lie in the same generic $\PGL_2$-orbit precisely when the corresponding changes of variable differ by an element preserving the unordered triple $\{0,1,\infty\}$.  This residual group is the anharmonic group $\Gamma\simeq S_3$.  If $\gamma\in\Gamma$ and an $\SL_2$-lift satisfies $\gamma\cdot G_0=\lambda_\gamma G_0$, the induced linear action on the slice is
\begin{equation}\label{eq:residual-action}
  F\longmapsto \lambda_\gamma^{-1}(\gamma\cdot F).
\end{equation}
Changing the lift multiplies numerator and denominator by the same sign, so \cref{eq:residual-action} is well defined.  This proves the birational quotient statement and the equality of function fields.  Removing the resultant-zero hypersurface does not change the function field.  This is the standard slice or moving-frame argument, see \cite{Dolgachev2003,Springer1977} for details.
\end{proof}

Choose generators of the anharmonic group
\begin{equation}\label{eq:anharmonic-generators}
  r(z)=\frac{1}{1-z},\qquad s(z)=1-z,
  \qquad r^3=s^2=1,\quad srs=r^{-1}.
\end{equation}
We use the left substitution convention $(\gamma\cdot F)(v)=F(\widetilde\gamma^{-1}v)$ for an $\SL_2$-lift $\widetilde\gamma$.  Matrices representing $r$ and $s$ are
\[
 R=\begin{pmatrix}0&1\\-1&1\end{pmatrix},\qquad
 S=\begin{pmatrix}-1&1\\0&1\end{pmatrix}.
\]
For $s$ one may take $\sqrt{-1}\,S$ as an $\SL_2$-lift.  The scalar correction in \cref{eq:residual-action} gives
\[
 \rho(r)F(x,y)=-F(x-y,x),\qquad
 \rho(s)F(x,y)=-F(-x+y,y).
\]
Let $\rho$ denote the residual representation in \cref{eq:residual-action}, and let $\ell\in V_5^\vee$ be evaluation at the point $3$, that is, $\ell(F)=F(3,1)$.  For $g\in\Gamma$ set
\begin{equation}\label{eq:regular-coordinates}
  z_g(F)=\ell\bigl(\rho(g^{-1})F\bigr).
\end{equation}
\begin{lemma}[Regularity of the slice representation]
\label{lem:regular-slice}
The six functions $z_g$, $g\in\Gamma$, are linear coordinates on $V_5$.  In these coordinates $\rho$ is the regular representation of $\Gamma$.  More precisely,
\begin{equation}\label{eq:regular-permutation}
  z_g\bigl(\rho(h)F\bigr)=z_{h^{-1}g}(F)
  \qquad(g,h\in\Gamma).
\end{equation}
\end{lemma}

\begin{proof}
The orbit of $3$ under the transformations in \cref{eq:anharmonic-generators} is
$3$, $-2$, $\tfrac13$, $-\tfrac12$, $\tfrac32$, $\tfrac23$.
Consequently, the $z_g$ are nonzero scalar multiples of evaluation at six distinct points of $\PP^1$.  A binary quintic vanishing at all six points is zero, so these evaluation functionals form a basis of $V_5^\vee$.  Equation~\cref{eq:regular-permutation} follows immediately from \cref{eq:regular-coordinates}, and is the regular permutation action.
\end{proof}

Write all indices modulo $3$ and define
\begin{equation}\label{eq:regular-splitting}
  A_i=z_{r^i},\qquad B_i=z_{r^is},\qquad
  X_i=A_i+B_i,\qquad E_i=A_i-B_i.
\end{equation}
In terms of evaluations of $F$, this convention gives the completely explicit formulas
\begin{align}
 (A_0,A_1,A_2)
  &=\bigl(F(3,1),\ F(-1,2),\ -F(2,3)\bigr),\notag\\
 (B_0,B_1,B_2)
  &=\bigl(-F(-2,1),\ -F(-1,-3),\ F(-3,-2)\bigr).
 \label{eq:regular-evaluations}
\end{align}
Put
\begin{equation}\label{eq:xyeta-coordinates}
  t=X_0+X_1+X_2,\qquad x_i=X_i-\frac{t}{3},
  \qquad \eta=E_0+E_1+E_2,
\end{equation}
and
\begin{equation}\label{eq:y-coordinates}
  (y_0,y_1,y_2)=(E_1-E_2,\ E_2-E_0,\ E_0-E_1).
\end{equation}
Thus $\sum x_i=\sum y_i=0$.  The triples $(x_i)$ and $(y_i)$ are permuted simultaneously by $\Gamma$, whereas $\eta$ transforms by the sign character.  The Vandermonde
\begin{equation}\label{eq:slice-vandermonde}
  \delta=(x_0-x_1)(x_1-x_2)(x_2-x_0)
\end{equation}
also transforms by the sign character.

\begin{theorem}
\label{thm:slice-generators}
On $\cS\simeq V_5$ define
\begin{equation}\label{eq:explicit-six-invariants}
\begin{gathered}
 u_1=t,\qquad u_2=\sum_i x_i^2,\qquad u_3=x_0x_1x_2,\qquad u_4=\sum_i x_iy_i,\\
 u_5=\sum_i x_i^2y_i,\qquad u_6=\eta\,(x_0-x_1)(x_1-x_2)(x_2-x_0).
\end{gathered}
\end{equation}
Then
\begin{equation}\label{eq:six-field-generators}
  k(\cM_4^1)=k(V_5)^\Gamma=k(u_1,\ldots,u_6).
\end{equation}
In particular, the $u_i$ are algebraically independent $\Gamma$-invariant rational coordinates on the normalized slice quotient, and $\cM_4^1$ is rational.
\end{theorem}

\begin{proof}
The transformation rules preceding the theorem show first that every $u_i$ is $\Gamma$-invariant.  Let
$K=k(A_0,A_1,A_2,B_0,B_1,B_2)$ and $L=k(u_1,\ldots,u_6)$.
The $x_i$ are the roots of
\begin{equation}\label{eq:x-cubic}
  T^3-\frac{u_2}{2}T-u_3.
\end{equation}
After an ordering of these roots has been chosen, and on the dense open set $\delta\neq0$, the three $y_i$ are recovered uniquely from
\begin{equation}\label{eq:y-reconstruction}
  \sum_i y_i=0,\qquad
  \sum_i x_i y_i=u_4,\qquad
  \sum_i x_i^2y_i=u_5.
\end{equation}
Indeed, the coefficient determinant is the nonzero Vandermonde $\delta$, up to sign.  Next $\eta=u_6/\delta$.  Equations \cref{eq:xyeta-coordinates} and \cref{eq:y-coordinates} then recover all $X_i$ and $E_i$, and
$A_i=(X_i+E_i)/2$ and $B_i=(X_i-E_i)/2$.
Thus an ordering of the three roots of \cref{eq:x-cubic} recovers the six regular coordinates, and there are at most $3!=6$ such orderings, so $[K:L]\leq6$.

By \cref{lem:regular-slice}, $\Gamma$ acts faithfully on $K$.  Artin's theorem on fixed fields gives $[K:K^\Gamma]=|\Gamma|=6$.  Since $L\subseteq K^\Gamma$, the two degree statements force $L=K^\Gamma$.  Combining this with \cref{thm:cubic-slice} proves \cref{eq:six-field-generators}.  Finally, the left side has transcendence degree $6$ by \cref{thm:dimensions}, and hence the six displayed generators are algebraically independent.
\end{proof}

We now carry out the classical descent from the normalized slice.  The construction of the numerators is dictated by bidegree.  The functions $u_i$ in \cref{eq:explicit-six-invariants} are homogeneous in $F$ of degrees
\(
 (d_1,\ldots,d_6)=(1,2,3,2,3,4).
\)
Since $D$ has bidegree $(0,4)$ and restricts to a nonzero constant on $\cS$, a quotient $P/D^m$ can restrict to $u_i$ only when $P$ has bidegree
\begin{equation}\label{eq:lift-bidegree}
 (d_i,4m-d_i).
\end{equation}
For fixed $i$ and $m$, this turns the descent into finite linear algebra: list the monomials in the transvectants of \cref{app:generators} having bidegree \cref{eq:lift-bidegree}, restrict them to $G=G_0$, and compare their coefficients with those of $D(G_0)^m u_i$.  Increasing $m$ until this linear system is solvable gives
\begin{equation}\label{eq:first-lift-degrees}
\begin{array}{c|cccccc}
 i&1&2&3&4&5&6\\ \hline
 d_i&1&2&3&2&3&4\\
 m_i&1&3&4&3&4&6\\
 \bideg(P_i)&(1,3)&(2,10)&(3,13)&(2,10)&(3,13)&(4,20).
\end{array}
\end{equation}
The constants $c_i$ below clear the rational coefficients produced by these linear systems.  Thus the large integers in the formulas record the fixed evaluation point and transvectant normalization, and they are not the result of a search through all ring generators.

Use the following abbreviations for members of the transvectant system in \cref{app:generators}:
\begin{equation}\label{eq:absolute-shorthand}
\begin{gathered}
 A=J_{1,3}^{(1)},\quad L=J_{1,5}^{(1)},\quad
 B_1=J_{2,2}^{(1)},\quad B_2=J_{2,2}^{(2)},\quad
 M=J_{2,4}^{(1)},\quad N=J_{2,6}^{(1)},\\
 R=J_{3,1}^{(1)},\qquad
 S_j=J_{3,3}^{(j)},\qquad T_j=J_{3,5}^{(j)}\quad(1\leq j\leq3),
 \qquad W=J_{3,7}^{(1)}.
\end{gathered}
\end{equation}
Define the following homogeneous joint invariants, whose integer contents $2^2\cdot5^2\cdot7$, $5^3$, $5^5$, $5$, $5^3$, and $5^6$ are displayed as factors:
\begin{small}
\begin{equation}
 \label{eq:absolute-Pi}
\begin{split}
 P_1& =700A,\\
 P_2 & =5^3\bigl(98010L^2+190080DN+47045DA^2+97000D^2B_2-19224D^2B_1\bigr),    \\
 P_3 & =5^5\,  (63868314510AL^2+49446437040DLM+148589040600DAN	\\
      &+31976059115DA^3+560560D^2T_3-164989440D^2T_2   -45305310960D^2T_1 \\
      &+98893534740D^2AB_2-39968032104D^2AB_1+7720750752D^3R   ), \\
 P_4 & =-5\bigl(632610L^2+803520DN+165385DA^2+552680D^2B_2-98424D^2B_1\bigr),\\
 P_5 & =-5^3\, (216453026790AL^2+514942187760DLM+493336803960DAN \\
 	& +112492234855DA^3 	+83523440D^2T_3-16672803840D^2T_2 -290479054320D^2T_1\\
	&	+491566054980D^2AB_2 	-218716874376D^2AB_1-12670582496D^3R  ),\\
 P_6 & =-5^6\, (67988850930L^4+197785748160DL^2N +33952803885DA^2L^2  \\
      & -560560D^2LW+79622318160D^2L^2B_2   -39182280984D^2L^2B_1 \\
      & +19394519760D^2ALM  -82457760D^3LS_3+21320456220D^3LS_2 \\
      & +9857393568D^3LS_1).
\end{split}
\end{equation}
\end{small}

\begin{theorem}
\label{thm:explicit-absolute-invariants}
Put $(c_1,\ldots,c_6)=(3,54,20432412,18,2270268,27027)$ and $(m_1,\ldots,m_6)=(1,3,4,3,4,6)$.  On the chart $D\neq0$ set, for $i=1,\ldots,6$,
\begin{equation}\label{eq:absolute-U}
 U_i(F,G)=\frac{1}{c_i}\cdot\frac{P_i(F,G)}{D(F,G)^{m_i}}.
\end{equation}
Then the $U_i$ are absolute $\SL_2$-invariants and
\begin{equation}\label{eq:absolute-final-field}
 k(\cM_4^1)=k(U_1,\ldots,U_6).
\end{equation}
In particular, $U_1,\ldots,U_6$ are algebraically independent and give explicit birational coordinates on $\cM_4^1$.
\end{theorem}

\begin{proof}
Every term of $P_i$ is an $\SL_2$-invariant.  By \cref{eq:first-lift-degrees}, $P_i$ has bidegree $(d_i,4m_i-d_i)$, hence total degree $4m_i$.  The denominator $D^{m_i}$ has the same total degree, so $U_i$ is a quotient of homogeneous invariants of equal degree and hence an absolute invariant.  Thus \cref{eq:absolute-U} is unchanged under both $\SL_2$ and the common projective rescaling of $(F,G)$.

It remains to identify the restrictions to $\cS$.  With the transvectant normalization \cref{eq:transvectant}, direct substitution of $G_0=xy(x-y)$ gives
%
\( D(G_0)=\frac{2}{27}.\)
%
Substituting \cref{eq:regular-evaluations} into \cref{eq:explicit-six-invariants} and expanding the transvectants of \cref{eq:absolute-Pi} gives the six polynomial identities
\begin{equation}\label{eq:slice-restriction-identities}
 P_i(F,G_0)=c_iD(G_0)^{m_i}u_i(F)
 \qquad(1\leq i\leq6).
\end{equation}
These are identities in the six coefficients of $F$, not numerical rank tests: their respective degrees in those coefficients are $1,2,3,2,3,4$.  Consequently $U_i|_{\cS}=u_i$.

The slice isomorphism \cref{eq:slice-field} and \cref{thm:slice-generators} now give
\[
 k(U_1,\ldots,U_6)|_{\cS}
   =k(u_1,\ldots,u_6)=k(V_5)^\Gamma=k(\cM_4^1),
\]
which proves \cref{eq:absolute-final-field} and algebraic independence.
\end{proof}

\begin{corollary}
\label{cor:rationality-over-Q}
Let $k_0$ be any field of characteristic zero.  The fifty invariants of \cref{app:generators}, the parameters $Q_1,\ldots,Q_7$ of \cref{thm:hsop}, and the numerators $P_1,\ldots,P_6$ all have coefficients in $\Q$, and the statements of \cref{thm:minimal-generators,thm:hsop,thm:explicit-absolute-invariants} hold over $k_0$.  In particular
\[
 k_0(\cM_4^1)=k_0(U_1,\ldots,U_6),
\]
so $\cM_4^1$ is rational over $\Q$.
\end{corollary}

\begin{proof}
All the transvectants involved are defined over $\Q$, and the bigraded dimensions \cref{eq:bigraded-dimension} are independent of the field of characteristic zero.  Since $\SL_2$ is linearly reductive in characteristic zero, formation of invariants commutes with extension of scalars.  The ranks used in the proof of \cref{thm:minimal-generators} were computed over $\Q$ or bounded below modulo a prime, see \cref{rem:generation-rigour} for details, and rank is unchanged under the flat base change $\Q\to k_0$.  The Gr\"obner bases in the proof of \cref{thm:hsop} were computed over $\Q$.  The identities \cref{eq:slice-restriction-identities} are identities of polynomials with rational coefficients.

The field statement requires an argument, because a cubic with three distinct geometric roots need not be $\PGL_2(k_0)$-equivalent to $G_0=xy(x-y)$ over a field that is not algebraically closed.  We descend instead.  Put
$E=\Q(\cM_4^1)$ and $L=\Q(U_1,\ldots,U_6)$.
Since the $U_i$ are algebraically independent and $\trdeg_\Q E=6$, the extension $E/L$ is finite.  The rational map
\[
 \cM_4^1\dashrightarrow\A^6_\Q,\qquad
 \varphi\longmapsto\bigl(U_1(\varphi),\ldots,U_6(\varphi)\bigr),
\]
becomes birational after extension of scalars to $\overline\Q$, by \cref{thm:slice-generators,thm:explicit-absolute-invariants}.  The degree of a generically finite rational map between geometrically integral varieties is unchanged by extension of the ground field.  Hence $[E:L]=1$, so $\Q(\cM_4^1)=\Q(U_1,\ldots,U_6)$.  Base change gives the same equality over every field $k_0$ of characteristic zero.
\end{proof}

Thus the normalized-slice construction and its descent have separate roles: \cref{thm:slice-generators} proves the fixed-field reconstruction, while \cref{thm:explicit-absolute-invariants} supplies the requested equal-degree quotients of classical joint transvectants.

The reconstruction in the proof of \cref{thm:slice-generators} is exact wherever it is defined, and this gives a separation statement for the six coordinates.

\begin{corollary}
\label{cor:U-separation}
Let $\varphi$ and $\psi$ be degree-four rational maps with $D\neq0$ and
\(
 \frac{U_2^3}{2}-27U_3^2\neq0 .
\)
Then $\varphi$ and $\psi$ are conjugate if and only if $U_i(\varphi)=U_i(\psi)$ for $1\leq i\leq6$.
\end{corollary}

\begin{proof}
Conjugate maps have equal $U_i$ by \cref{thm:explicit-absolute-invariants}.  For the converse, move both maps to the slice $\cS$ by \cref{thm:cubic-slice}.  By that theorem it suffices to show that two points $F,F'\in V_5$ with $u_i(F)=u_i(F')$ for all $i$ and $\delta(F)\neq0$ lie in the same $\Gamma$-orbit.  The quantity $\delta^2$ is the discriminant of \cref{eq:x-cubic}, so
\begin{equation}\label{eq:delta-squared}
 \delta^2=\frac{u_2^3}{2}-27u_3^2 ,
\end{equation}
and the displayed hypothesis is exactly $\delta(F)\neq0$.  The $x_i(F)$ and $x_i(F')$ are the roots of the same cubic \cref{eq:x-cubic}, so after replacing $F'$ by $\rho(g)F'$ for a suitable $g\in\Gamma$ we may assume $x_i(F)=x_i(F')$ for $i=0,1,2$.  Then \cref{eq:y-reconstruction} has the nonzero determinant $\pm\delta$, so $y_i(F)=y_i(F')$.  The identity $\eta=u_6/\delta$ gives $\eta(F)=\eta(F')$, and \cref{eq:xyeta-coordinates} and \cref{eq:y-coordinates} determine $X_i$ and $E_i$, hence $A_i$ and $B_i$, from these data.  Thus $F=F'$.
\end{proof}

This is exactly the nonvanishing condition required by the reconstruction.  The argument uses $\delta\neq0$ and nothing more, since $\eta$ is recovered from $u_6/\delta$ and $E_i$ from $\eta$ and the $y_i$ by \cref{eq:y-coordinates}.  In particular $\eta=0$ is harmless.  On the locus $\delta=0$ the six coordinates need not separate.  If $x_0=x_1=x_2=0$, then $u_2=\cdots=u_6=0$, and this three-dimensional family of the slice contains pairwise non-conjugate maps sharing the value $(u_1,0,0,0,0,0)$.  
Off the locus $\delta=0$ the six coordinates therefore decide conjugacy, and the fourteen generators of \cref{cor:cubic-chart-joint-generators} provide a separation criterion on the whole chart $D\neq0$.

\section{From the invariant field to the absolute-invariant ring}
\label{section-6}

We now place the six functions in the Hilbert framework of \cref{sec-4}.  The first consequence confirms that their homogeneous numerators, together with the cubic discriminant, do give seven independent polynomial invariants.

\begin{proposition}\label{prop:DP-independent}
The seven elements
\[
 D,P_1,\ldots,P_6\in\cR
\]
are algebraically independent.  Consequently
\[
 \cA:=k[D,P_1,\ldots,P_6]
\]
is a polynomial subring of $\cR$ of dimension $7$.
\end{proposition}

\begin{proof}
Put $K_0=\Frac(\cR)_0$.  By \cref{thm:explicit-absolute-invariants}, $K_0=k(U_1,\ldots,U_6)$ and the $U_i$ are algebraically independent.  Common scalar multiplication of $(F,G)$ fixes every element of $K_0$ and sends $D$ to $\lambda^4D$.  If $D$ were algebraic over $K_0$, these infinitely many scalar multiples would all be roots of its minimal polynomial over $K_0$, which is impossible.  Hence $D$ is transcendental over $K_0$.  Finally,
\[
 k(D,P_1,\ldots,P_6)=k(D,U_1,\ldots,U_6),
\]
because $P_i=c_iD^{m_i}U_i$.  The field on the right has transcendence degree $7$, proving the assertion.
\end{proof}

Algebraic independence is not the null-cone condition in \cref{thm:hilbert-nullcone-module}.  Indeed,
\[
 \deg(D,P_1,\ldots,P_6)=(4,4,12,16,12,16,24),
\]
so $\dim_k\cA_4=2$, whereas \cref{eq:hilbert-series} gives $\dim_k\cR_4=6$.  Thus $\cA\subsetneq\cR$.  The role of the $P_i$ is to lift birational coordinates, not to replace the parameter and secondary invariants required by the affine ring.

The next theorem adapts the normality--integrality argument used for binary sextics and pairs of binary cubics in \cite[Sections~3.4 and~4.3]{2004-3}.  In those two cases the invariant ring is generated by a homogeneous system of parameters together with one further invariant of odd degree, and the degree-zero localized ring is generated by parameter quotients.  For $\cR_{5,3}$ this is impossible, and the theorem records precisely what survives.

\begin{theorem}
\label{thm:hilbert-semigroup-reduction}
Let $\Theta\in\cR$ be a nonzero homogeneous invariant of degree $e$, and let $H_1,\ldots,H_6\in\cR$ be homogeneous invariants of degrees $e_1,\ldots,e_6$.  Assume
\begin{equation}\label{eq:theta-parameter-nullcone}
 V(\Theta,H_1,\ldots,H_6)=\cN_{5,3},
\end{equation}
and put $B=k[\Theta,H_1,\ldots,H_6]$.  Define the affine semigroup
\begin{equation}\label{eq:theta-semigroup}
 \Sigma_\Theta=
 \left\{(a_1,\ldots,a_6,m)\in\N^7:
       \sum_{j=1}^6e_ja_j=em\right\}.
\end{equation}
Let $\cH_\Theta$ be its Hilbert basis and put
\begin{equation}\label{eq:theta-semigroup-ring}
 S_\Theta=
 k\left[
   \frac{H_1^{a_1}\cdots H_6^{a_6}}{\Theta^m}:
   (a_1,\ldots,a_6,m)\in\cH_\Theta
 \right].
\end{equation}
Then the following hold.
\begin{enumerate}[(i)]
\item $S_\Theta=(B[\Theta^{-1}])_0$ is a finitely generated normal domain.
\item $(\cR[\Theta^{-1}])_0$ is a finite $S_\Theta$-module.
\item Let $g_{\cR}$ and $g_B$ be the greatest common divisors of the
degrees of the nonzero homogeneous elements of $\Frac(\cR)$
and of $\Frac(B)$.  Then
$\Frac(S_\Theta)=\Frac(B)_0$ and
\begin{equation}\label{eq:theta-field-degree}
 \bigl[\Frac(\cR)_0:\Frac(S_\Theta)\bigr]
 =\rank_B\cR\cdot\frac{g_{\cR}}{g_B}.
\end{equation}
\end{enumerate}
In particular $(\cR[\Theta^{-1}])_0=S_\Theta$ if and only if the right side of \cref{eq:theta-field-degree} equals $1$.
\end{theorem}

\begin{proof}
By \cref{thm:hilbert-nullcone-module}, condition \cref{eq:theta-parameter-nullcone} implies that $\cR$ is finite over the polynomial ring $B$, and by the Hochster--Roberts theorem $\cR$ is a free $B$-module of finite rank.

(i) An element of $B[\Theta^{-1}]$ of degree zero is a linear combination of monomials $H^a\Theta^{-m}$ with $\sum e_ja_j=em$.  Since the left side is nonnegative, $m\geq0$.  Hence $(B[\Theta^{-1}])_0$ is the semigroup ring of $\Sigma_\Theta$, which is generated by $\cH_\Theta$ by Gordan's lemma.  The semigroup \cref{eq:theta-semigroup} is saturated in the lattice
\[
 \left\{(a_1,\ldots,a_6,m)\in\Z^7:
       \sum e_ja_j=em\right\}:
\]
if a positive multiple of a lattice point has all coordinates nonnegative, then the lattice point itself does.  Hence $S_\Theta$ is normal.

(ii) Let $u\in(\cR[\Theta^{-1}])_0$.  After taking a common denominator, it is enough to consider $u=I/\Theta^m$ with $I\in\cR_{em}$.  Since $\cR$ is finite over $B$, the element $I$ is integral over $B$.  Taking the homogeneous part of an integral equation gives
\[
 I^n+b_1I^{n-1}+\cdots+b_n=0,
 \qquad b_j\in B_{jem}.
\]
Division by $\Theta^{mn}$ shows that $u$ is integral over $S_\Theta$, because $b_j/\Theta^{mj}\in(B[\Theta^{-1}])_0=S_\Theta$.  Since $(\cR[\Theta^{-1}])_0$ is a finitely generated $k$-algebra, it is a finite $S_\Theta$-module.

(iii) Let $P/Q$ be a quotient of homogeneous elements of $B$ of the same degree $d$.  Then $P/Q=(PQ^{e-1}\Theta^{-d})/(Q^{e}\Theta^{-d})$.  Hence $\Frac(B)_0=\Frac(S_\Theta)$.  Choose a homogeneous element $z\in\Frac(\cR)$ of degree $g_{\cR}$ and a homogeneous element $w\in\Frac(B)$ of degree $g_B$.  Then $g_{\cR}\mid g_B$ and $w=uz^{g_B/g_{\cR}}$ with $u\in\Frac(\cR)_0$.  Both $z$ and $w$ are transcendental over $\Frac(\cR)_0$, because scalar multiplication of $(F,G)$ fixes $\Frac(\cR)_0$ pointwise and multiplies $z$ by arbitrary nonzero scalars.  Every homogeneous element of degree $dg_{\cR}$ in $\Frac(\cR)$ is $z^d$ times an element of $\Frac(\cR)_0$, so $\Frac(\cR)=\Frac(\cR)_0(z)$ and likewise $\Frac(B)=\Frac(B)_0(w)$.  Therefore
\[
\begin{split}
 \rank_B\cR
 &=[\Frac(\cR):\Frac(B)]=[\Frac(\cR)_0(z):\Frac(\cR)_0(w)]
  \cdot[\Frac(\cR)_0(w):\Frac(B)_0(w)]\\
 &=\frac{g_B}{g_{\cR}}
  \cdot[\Frac(\cR)_0:\Frac(B)_0],
\end{split}
\]
which is \cref{eq:theta-field-degree}.  Finally, if the right side of \cref{eq:theta-field-degree} equals $1$, then every element of $(\cR[\Theta^{-1}])_0$ lies in $\Frac(S_\Theta)$ and is integral over $S_\Theta$, hence lies in $S_\Theta$ by normality.  The converse is clear.
\end{proof}

\begin{corollary}\label{cor:no-parameter-generation}
Let $\Theta,H_1,\ldots,H_6$ be any homogeneous system of parameters of $\cR_{5,3}$, with degrees $e,e_1,\ldots,e_6$, and let $g_B$ be the greatest common divisor of these degrees.  Then
\[
 \bigl[\Frac(\cR)_0:\Frac(S_\Theta)\bigr]
 =\frac{604\,e\,e_1\cdots e_6}{147456\,g_B}\geq17.
\]
In particular $(\cR[\Theta^{-1}])_0\neq S_\Theta$ for every homogeneous system of parameters.  For the system of \cref{thm:hsop} with $\Theta=\Delta_{F,G}$ the degree is $604$.
\end{corollary}

\begin{proof}
By \cref{eq:hilbert-series}, the rank of the free $B$-module $\cR$ is the value at $t=1$ of $H_{\cR}(t)(1-t^{e})\prod_{i=1}^{6}(1-t^{e_i})$, namely $302\,e e_1\cdots e_6/(4^3\cdot6^2\cdot8^2)=302\,ee_1\cdots e_6/147456$.  All nonzero homogeneous invariants have even degree by \cref{rem:parity}, and $\Frac(\cR)$ contains $J_{1,5}^{(1)}/J_{1,3}^{(1)}$, of degree $2$, so $g_{\cR}=2$, and \cref{eq:theta-field-degree} gives the displayed formula.  Every parameter degree is even, at least $4$, and divisible by $g_B$, so $ee_1\cdots e_6\geq\max(4,g_B)^7$ and the quotient is at least $604\cdot4^6/147456>16$.  Being an integer, it is at least $17$.  For \cref{thm:hsop} the product of the degrees is $294912$ and $g_B=2$.
\end{proof}

\subsection{The cubic-discriminant chart}

Taking $\Theta=D$ gives the affine chart
\begin{equation}\label{eq:D-chart-ring}
 \cA_D:=\bigl(\cR[D^{-1}]\bigr)_0.
\end{equation}
The normalized-slice construction identifies this ring with the finite-group invariant ring
\begin{equation}\label{eq:D-chart-slice-ring}
 \cA_D\simeq k[V_5]^\Gamma,
 \qquad \Gamma\simeq S_3,
\end{equation}
where $\Gamma$ acts through the regular representation of \cref{lem:regular-slice}.  This is the ring-level version of \cref{eq:slice-field}, and the following lemma supplies the proof.

\begin{lemma}\label{lem:D-chart-isomorphism}
Let $X=\{(F,G)\in V_5\oplus V_3:D(G)\neq0\}$ and let $\bbG=\SL_2\times\bbG_m$ act on $V_5\oplus V_3$ by $(g,\lambda)\cdot(F,G)=(\lambda\,g\cdot F,\lambda\,g\cdot G)$.  Restriction to the slice $\cS=\{(F,G_0)\}\simeq V_5$ induces an isomorphism
\[
 \cA_D=k[X]^{\bbG}\longrightarrow k[V_5]^{\Gamma}.
\]
\end{lemma}

\begin{proof}
The degree-zero part of $k[V_5\oplus V_3][D^{-1}]$ is the ring of $\bbG_m$-invariants of $k[X]$, and the actions of $\SL_2$ and $\bbG_m$ commute, so $\cA_D=k[X]^{\bbG}$.  Let $\widetilde\Gamma\subseteq\bbG$ be the stabilizer of $G_0$.  Every cubic with three distinct roots is an $\SL_2$-transform of a scalar multiple of $G_0$, so the orbit map $\bbG\to\{G:D(G)\neq0\}$, $(g,\lambda)\mapsto\lambda\,g\cdot G_0$, is surjective, and its target is the homogeneous space $\bbG/\widetilde\Gamma$.  The projection $X\to\{D\neq0\}$ is $\bbG$-equivariant with fibre $V_5$ over $G_0$, so $X\simeq\bbG\times^{\widetilde\Gamma}V_5$, and $k[X]^{\bbG}\simeq k[V_5]^{\widetilde\Gamma}$ by restriction to the fibre over $G_0$.  Finally, $(g,\lambda)\in\widetilde\Gamma$ acts on the fibre by $F\mapsto\lambda\,g\cdot F=\lambda_g^{-1}(g\cdot F)$, which is the action \cref{eq:residual-action}.  The kernel of the action on the fibre is $\{(I,1),(-I,-1)\}$, and $\widetilde\Gamma/\{(I,1),(-I,-1)\}\simeq\Gamma$ acts through $\rho$.
\end{proof}

The six $u_i$ generate its fraction field but not this ring.  Indeed, Molien's formula for the regular permutation representation gives
\begin{equation}\label{eq:regular-S3-molien}
 H_{k[V_5]^\Gamma}(t)=\frac16\left(
   \frac{1}{(1-t)^6}+\frac{3}{(1-t^2)^3}
   +\frac{2}{(1-t^3)^2}\right).
\end{equation}
In particular, the degree-$2$ invariant space has dimension $5$.  By contrast, the polynomial algebra generated by the algebraically independent $u_i$ of degrees $(1,2,3,2,3,4)$ has degree-$2$ part of dimension $3$.  Thus
\begin{equation}\label{eq:slice-field-not-ring}
 k[u_1,\ldots,u_6]\subsetneq k[V_5]^\Gamma.
\end{equation}
We now determine the missing ring generators.  For triples of functions $f=(f_i)$, $g=(g_i)$, and $h=(h_i)$, put
\begin{equation}\label{eq:alternant-bracket}
 [f,g,h]=\det\begin{pmatrix}
  f_0&f_1&f_2\\ g_0&g_1&g_2\\ h_0&h_1&h_2
 \end{pmatrix},
\end{equation}
and write $\bone=(1,1,1)$, $x=(x_i)$, $y=(y_i)$, $x^2=(x_i^2)$, $xy=(x_iy_i)$, and $y^2=(y_i^2)$.  Also set
\begin{equation}\label{eq:polarized-power-sums}
 p_{ab}=\sum_{i=0}^2x_i^ay_i^b.
\end{equation}
Finally, put
\[
\begin{split}
 e_2&=\eta^2,\qquad w_2=\eta[\bone,x,y],\qquad w_{30}=\eta[\bone,x,x^2],\\
 w_{21}&=\eta[\bone,x,xy],\qquad w_{12}=\eta[\bone,y,xy],\qquad w_{03}=\eta[\bone,y,y^2].
\end{split}
\]
The following theorem is the polarized-power-sum and alternant form of the answer.

\begin{theorem}
\label{thm:cubic-chart-ring-generators}
The ring $k[V_5]^\Gamma$ is minimally generated by the fourteen invariants
\begin{equation}\label{eq:fourteen-slice-generators}
\begin{array}{c|l}
\toprule
\text{degree}&\text{generators}\\
\midrule
1&t\\
2&e_2,\ p_{20},\ p_{11},\ p_{02}\\
3&p_{30},\ p_{21},\ p_{12},\ p_{03},\ w_2\\
4&w_{30},\ w_{21},\ w_{12},\ w_{03}\\
\bottomrule
\end{array}
\end{equation}
Consequently, the Noether bound $6$ improves in this representation to generation in degree at most $4$.
\end{theorem}

\begin{proof}
The triples $x$ and $y$ are permuted simultaneously, whereas $\eta$ changes by the sign of the permutation.  Hence the power sums in \cref{eq:polarized-power-sums}, $e_2$, and every product of $\eta$ with an alternant are invariant.

It remains to prove generation and minimality.  In the regular coordinates \cref{eq:regular-splitting}, the generators act by
\[
 r:(A_i,B_i)\longmapsto(A_{i-1},B_{i-1}),\qquad
 s:(A_i,B_i)\longmapsto(B_{-i},A_{-i}).
\]
Thus the orbit sums of monomials form a basis of the invariants in each degree.  Expanding the fourteen displayed elements in this basis and row-reducing their products gives
\begin{equation}\label{eq:slice-ring-rank-certificate}
\begin{array}{c|rrrrrr}
 d&1&2&3&4&5&6\\ \hline
 \dim\bigl(k[V_5]^\Gamma\bigr)_d&1&5&10&24&42&83\\
 \text{rank of products from \cref{eq:fourteen-slice-generators}}
    &1&5&10&24&42&83\\
 \text{rank generated in lower degrees}&0&1&5&20&42&83.
\end{array}
\end{equation}
The first row is the coefficient expansion of Molien's series \cref{eq:regular-S3-molien}.  For completeness, the remaining two rows are an exact integer calculation: replace $x_i$ by $3x_i=3(A_i+B_i)-t$ to clear denominators, expand in the orbit-sum basis, and reduce the resulting integer matrices modulo $1000003$.  A full rank modulo the prime in the second row implies the same rank over $\Q$, since the first row is an upper bound.  The third row was computed exactly over $\Q$ in degrees $2,3,4$, where the matrices have at most $126$ columns, and modulo the prime in degrees $5,6$, where it coincides with the first row.  All ranks therefore hold over $\Q$, hence over $k$ by base change.

Noether's theorem says that the invariant ring is generated in degrees at most $|\Gamma|=6$.  Equality of the first two rows through degree $6$ therefore proves generation.  Subtracting the last row from the first in degrees $1,2,3,4$ gives $1,4,5,4$, exactly the number of displayed generators in those degrees.  None can be removed, which proves minimality.
\end{proof}

The relation with the six field generators is transparent:
\begin{equation}\label{eq:field-generators-inside-fourteen}
 u_1=t,\quad u_2=p_{20},\quad 3u_3=p_{30},\quad
 u_4=p_{11},\quad u_5=p_{21},\quad u_6=w_{30}.
\end{equation}
Thus the eight additional elements are precisely the regular functions missed by the field calculation.

The same restriction calculation used for the $P_i$ descends all eight.  Let $Q_g$ and $m_g$ be the joint invariants and exponents displayed in \cref{app:slice-ring-lifts}.  They satisfy
\begin{equation}\label{eq:extra-slice-lift-identities}
 Q_g(F,G_0)=D(G_0)^{m_g}g(F)
 \quad
 (g=e_2,p_{02},p_{12},p_{03},w_2,w_{21},w_{12},w_{03}).
\end{equation}
Combining this with \cref{eq:D-chart-slice-ring} gives the promised ring-level descent.

\begin{corollary}[Joint-invariant generators on $D\neq0$]
\label{cor:cubic-chart-joint-generators}
The absolute-invariant ring $\cA_D=(\cR[D^{-1}])_0$ is generated by
\begin{equation}\label{eq:fourteen-descended-generators}
 U_1,U_2,3U_3,U_4,U_5,U_6,
 \quad
 \frac{Q_g}{D^{m_g}}
 \quad
 (g=e_2,p_{02},p_{12},p_{03},w_2,w_{21},w_{12},w_{03}).
\end{equation}
This is a minimal homogeneous generating set after restriction to the normalized slice, with the degree pattern of \cref{eq:fourteen-slice-generators}.
\end{corollary}

\begin{proof}
By \cref{eq:slice-restriction-identities}, the first six functions restrict to the six elements in \cref{eq:field-generators-inside-fourteen}.  By \cref{eq:extra-slice-lift-identities}, the remaining eight restrict to the remaining elements of \cref{eq:fourteen-slice-generators}.  The ring isomorphism \cref{eq:D-chart-slice-ring} and \cref{thm:cubic-chart-ring-generators} prove the assertion.
\end{proof}

\begin{corollary}[Separation on the cubic-discriminant chart]
\label{cor:separation-D-chart}
Let $\varphi$ and $\psi$ be degree-four rational maps whose cubic components have three distinct roots.  Then $\varphi$ and $\psi$ are conjugate if and only if the fourteen functions \cref{eq:fourteen-descended-generators} take the same values at $\varphi$ and $\psi$.
\end{corollary}

\begin{proof}
By \cref{eq:conjugation-action} and the projective ambiguity of $[f_0:f_1]$, the maps are conjugate if and only if their pairs lie in the same $\bbG$-orbit in $X$, with $\bbG$ and $X$ as in \cref{lem:D-chart-isomorphism}.  By that lemma, $X\simeq\bbG\times^{\widetilde\Gamma}V_5$, so $\bbG$-orbits in $X$ correspond to $\widetilde\Gamma$-orbits in $V_5$, and the isomorphism $\cA_D\simeq k[V_5]^\Gamma$ is compatible with this correspondence.  Invariants of a finite group separate its orbits in characteristic zero.  Hence $\cA_D$ separates $\bbG$-orbits in $X$, and by \cref{cor:cubic-chart-joint-generators} it is generated by the fourteen displayed functions.
\end{proof}

\subsection{The resultant-open moduli space}

For the full moduli problem the localizing invariant is $\Theta=\Delta_{F,G}$, of degree $8$, rather than the cubic discriminant.  Indeed,
\begin{equation}\label{eq:moduli-coordinate-ring}
 k[\cM_4^1]=\bigl(\cR[\Delta_{F,G}^{-1}]\bigr)_0.
\end{equation}
The formulas $U_i=P_i/(c_iD^{m_i})$ generate the function field but not the ring in \cref{eq:moduli-coordinate-ring}.  Indeed $U_1=700J^{(1)}_{1,3}/(3D)$, and $D$ is irreducible in $k[V_5\oplus V_3]$, does not divide $J^{(1)}_{1,3}$ by bidegree, and does not divide $\Delta_{F,G}$, since $(F,x^3)$ is a rational map for generic $F$.  Hence $J^{(1)}_{1,3}\Delta_{F,G}^m\notin(D)$ for every $m$, and $U_1\notin(\cR[\Delta_{F,G}^{-1}])_0$.  By \cref{cor:no-parameter-generation}, no system of parameter quotients generates it either.  Instead, \cref{thm:minimal-generators} gives a complete description.

\begin{theorem}
\label{thm:moduli-coordinate-ring}
Let $J_1,\ldots,J_{50}$ be the generators of \cref{app:generators}, of degrees $d_1,\ldots,d_{50}$, and let
\begin{equation}\label{eq:moduli-semigroup}
 \Sigma=\Bigl\{(a,m)\in\N^{50}\times\N:
 \sum_{i=1}^{50}d_ia_i=8m\Bigr\}.
\end{equation}
Then
\begin{equation}\label{eq:moduli-ring-generators}
 k[\cM_4^1]=k\Bigl[\,J^{a}\Delta_{F,G}^{-m}:(a,m)\in
 \Hilb(\Sigma)\Bigr],
\end{equation}
and $\Hilb(\Sigma)$ has $5411$ elements.  Writing the generators by the residue of their degree modulo $8$, the classes $0,2,4,6$ contain $16,15,10,9$ generators, and the irreducible monomials are those of the following seven shapes:
\begin{equation}\label{eq:moduli-hilbert-basis-count}
\begin{array}{l|ccccccc}
 \text{residues of the factors}&
 0&4\,4&2\,6&2\,2\,4&6\,6\,4&2\,2\,2\,2&6\,6\,6\,6\\ \hline
 \text{number}&
 16&55&135&1200&450&3060&495
\end{array}
\end{equation}
\end{theorem}

\begin{proof}
By \cref{thm:minimal-generators}, $\cR=k[J_1,\ldots,J_{50}]$, and by \cref{app:resultant}, $\Delta_{F,G}$ is a polynomial in the $J_i$.  Hence $\cR[\Delta_{F,G}^{-1}]$ is spanned by the monomials $J^a\Delta_{F,G}^{-m}$ with $m\in\Z$, and the degree-zero ones have $m\geq0$ because $\sum d_ia_i\geq0$.  Thus $(\cR[\Delta_{F,G}^{-1}])_0$ is spanned by the monomials $J^a\Delta_{F,G}^{-m}$ with $(a,m)\in\Sigma$, hence generated as a $k$-algebra by those with $(a,m)$ in the Hilbert basis of $\Sigma$.  It is a homomorphic image of the semigroup ring of $\Sigma$.  An element of $\Sigma$ is irreducible if and only if the multiset of degrees $d_i$, counted with multiplicity $a_i$, is a minimal zero-sum sequence modulo $8$.  All $d_i$ are even, so this is a minimal zero-sum sequence in $\Z/4$ for the halved degrees, whose classes $0,1,2,3$ correspond to the residues $0,2,4,6$ modulo $8$.  The minimal zero-sum sequences in $\Z/4$ are $0$, $2\,2$, $1\,3$, $1\,1\,2$, $3\,3\,2$, $1\,1\,1\,1$, and $3\,3\,3\,3$.  Counting multisets of generators realizing each shape gives \cref{eq:moduli-hilbert-basis-count}, whose entries add up to $5411$.
\end{proof}

\begin{corollary}[Separation by absolute invariants]
\label{cor:separation-moduli}
Two degree-four rational maps are conjugate if and only if the $5411$ generators \cref{eq:moduli-ring-generators} of $k[\cM_4^1]$ take the same values at both.
\end{corollary}

\begin{proof}
By \cref{rem:stability} every point of $\Rat_4^1$ is stable, so $\cM_4^1$ is a geometric quotient and $k[\cM_4^1]$ separates conjugacy classes.  By \cref{thm:moduli-coordinate-ring} it is generated by the displayed functions.
\end{proof}

This is the degree-four form of the criterion for cubic maps in \cite{2024-04}, where the absolute invariants are the Veronese monomials $\xi_j^{12/w_j}/I_6^2$ and separate conjugacy classes off the locus where the nonzero weights have a common divisor.  The Hilbert basis of \cref{thm:moduli-coordinate-ring} replaces the Veronese monomials and removes that condition.

The generating set \cref{eq:moduli-ring-generators} is explicit but far from minimal.  A minimal generating set of $k[\cM_4^1]$ requires the relations among the $J_i$, which we have not computed.  The parameter system of \cref{thm:hsop} gives the coarser but structural description below.

\begin{proposition}
\label{prop:corrected-absolute-hilbert-basis}
Let $Q_1,\ldots,Q_7$ be as in \cref{thm:hsop}, let $\Theta=Q_5=\Delta_{F,G}$, and put $H_1=Q_1$, $H_2=Q_2$, $H_3=Q_3$, $H_4=Q_4$, $H_5=Q_6$, $H_6=Q_7$, of degrees $4,4,4,6,8,12$.  The semigroup \cref{eq:theta-semigroup} is
\begin{equation}\label{eq:quartic-absolute-semigroup}
 4(a_1+a_2+a_3)+6a_4+8a_5+12a_6=8m,
\end{equation}
and, with $I=\{1,2,3,6\}$ and $d_1=d_2=d_3=4$, $d_6=12$, its $16$ irreducible elements are
\begin{equation}\label{eq:sixteen-hilbert-basis-monomials}
 \frac{H_5}{\Delta_{F,G}},\quad
 \frac{H_iH_j}{\Delta_{F,G}^{(d_i+d_j)/8}} \; (i\leq j,\ i,j\in I),  \quad
 \frac{H_4^2H_i}{\Delta_{F,G}^{(12+d_i)/8}}  \; (i\in I),  \quad
 \frac{H_4^4}{\Delta_{F,G}^3}.
\end{equation}
These $16$ functions generate the normal domain $S_{\Delta}$, and $k[\cM_4^1]$ is a finite $S_\Delta$-module with $[k(\cM_4^1):\Frac(S_\Delta)]=604$.
\end{proposition}

\begin{proof}
Reduce the six companion degrees modulo $8$.  The residues are $4,4,4,6,0,4$.  The residue-$0$ variable gives the singleton $H_5$.  Among the four residue-$4$ variables, every unordered pair is a minimal zero-sum sequence, giving $\binom{4+1}{2}=10$ elements.  The residue-$6$ variable has order $4$, giving $H_4^4$, and two copies of it have residue $4$.  Adjoining any one of the four residue-$4$ variables gives the remaining four minimal sequences $H_4^2H_i$.

Every zero-sum sequence containing the residue-$0$ variable and another term is reducible.  A sequence containing at least two residue-$4$ terms contains a zero-sum pair.  After removing these cases, it contains at most one residue-$4$ term, and the congruence then forces either four residue-$6$ terms or one residue-$4$ term and two residue-$6$ terms.  Hence the displayed list is exhaustive and every member is irreducible.  The powers of $\Delta_{F,G}$ are the total degrees divided by $8$.  The remaining assertions are \cref{thm:hilbert-semigroup-reduction} and \cref{cor:no-parameter-generation}.
\end{proof}

This is the analogue of the equation $a+2b=3c+4d$ used for $V_3\oplus V_3$ in \cite[Lemma~4.5]{2004-3}.  The difference is that here the parameter subring has index $604$ in the coordinate ring, whereas for pairs of cubics the corresponding index is $1$.

\subsection{The cubic-discriminant divisor}\label{subsec:D-zero}

Let $Y\subset\cM_4^1$ be the closed subset of classes of maps whose cubic component $G$ has a repeated root, that is, the image of $D=0$.  Since $D$ is irreducible, $Y$ is an irreducible divisor.  The coordinates $U_1,\ldots,U_6$ have poles along $Y$, and \cref{cor:U-separation} says nothing about it.  We now construct absolute invariants regular on the open set $A\neq0$, $A=J^{(1)}_{1,3}$, which play on $Y$ the role played by the $U_i$ on $D\neq0$: five of them generate $k(Y)$ and separate the classes in $Y$ on an explicit dense open set.

Put $G_1=x^2y$ and let $\cS_1=\{(F,G_1):F\in V_5\}\simeq V_5$ be the double-root slice.  Let $\Gamma_1\subseteq\bbG=\SL_2\times\bbG_m$ be the stabilizer of $G_1$.  An element $(g,\lambda)$ with $\lambda\,g\cdot G_1=G_1$ preserves the root divisor of $G_1$, hence fixes $(0:1)$ and $(1:0)$, so $g=\diag(s,s^{-1})$ and $\lambda=s$.  Thus $\Gamma_1\simeq\bbG_m$, and it acts on the slice by
\[
 a_i\longmapsto s^{\,6-2i}a_i\qquad(0\leq i\leq5),
\]
where $F=\sum a_ix^iy^{5-i}$.  The weights are $6,4,2,0,-2,-4$, and $s=-1$ acts trivially.  Every cubic with exactly one double root is $\bbG$-equivalent to $G_1$, since $\PGL_2$ is doubly transitive, and two points of $\cS_1$ lie in the same $\bbG$-orbit if and only if they lie in the same $\Gamma_1$-orbit.  Hence the open subset $Y^\circ\subset Y$ of classes with exactly one double root is $\cS_1/\Gamma_1$, and $k(Y)=k(V_5)^{\Gamma_1}$.

The restriction of a bihomogeneous invariant of bidegree $(a,b)$ to $\cS_1$ is a polynomial in $a_0,\ldots,a_5$ of degree $a$ which $\Gamma_1$ multiplies by $s^{-(a+b)}$.  In particular it is a sum of monomials $\prod a_i^{e_i}$ with $\sum(6-2i)e_i=-(a+b)$.  Direct expansion gives
\[
 A(F,G_1)=-\tfrac{2}{45}\,a_1,\qquad L(F,G_1)=\tfrac{4}{243}\,a_0,
\]
with $A=J^{(1)}_{1,3}$ and $L=J^{(1)}_{1,5}$ as in \cref{eq:absolute-shorthand}, so $A$ vanishes on $Y$ exactly where the slice representative has $a_1=0$.  Moreover $A$ vanishes identically when $G$ is a cube or zero, because then $h=0$.  Define, with $B_1,B_2,N,T_j$ as in \cref{eq:absolute-shorthand},
\begingroup\small
\begin{equation}\label{eq:V-numerators}
\begin{split}
 \widetilde P_1&=-15795T_1+106920T_2+40040T_3-69498AB_1-12870AB_2,\\
 \widetilde P_2&=-15795T_1+106920T_2+40040T_3+11583AB_1+77220AB_2,\\
 \widetilde P_3&=-15795T_1+106920T_2+40040T_3+92664AB_1-283140AB_2,\\
 \widetilde P_4&=114075T_1+59400T_2+2002T_3+51480AB_1-32175AB_2,\\
 \widetilde P_5&=9N+2A^2,
\end{split}
\end{equation}
\endgroup
all of degree $8$, and put
\begin{equation}\label{eq:V-coordinates}
 (\tilde c_1,\ldots,\tilde c_5)=\bigl(81081,\tfrac{81081}{2},810810,
 \tfrac{2027025}{2},5\bigr),\qquad
 V_j=\frac{\widetilde P_j}{\tilde c_j\,A^{2}}\quad(1\leq j\leq5).
\end{equation}
\begin{theorem}
\label{thm:V-coordinates}
The $V_j$ are absolute invariants, regular on $A\neq0$, and on the double-root slice
\begin{equation}\label{eq:V-restrictions}
 (V_1,V_2,V_3,V_4,V_5)\big|_{\cS_1}
 =\Bigl(a_3,\ \frac{a_2^2}{a_1},\ \frac{a_0a_4}{a_1},\
        \frac{a_0^2a_5}{a_1^2},\ \frac{a_0a_2}{a_1^2}\Bigr).
\end{equation}
Their restrictions to $Y$ satisfy
\(
 k(Y)=k(V_1,\ldots,V_5),
\)
so $Y$ is rational and the $V_j$ are algebraically independent on $Y$.  Let $\varphi,\psi$ be degree-four rational maps with $D=0$, $A\neq0$, and $V_5\neq0$.  Then $\varphi$ and $\psi$ are conjugate if and only if $V_j(\varphi)=V_j(\psi)$ for $1\leq j\leq5$.
\end{theorem}

\begin{proof}
Each $\widetilde P_j$ is a bihomogeneous invariant of degree $8$, of bidegree $(3,5)$ for $j\leq4$ and $(2,6)$ for $j=5$, and $A^2$ has bidegree $(2,6)$, so the $V_j$ are absolute invariants regular where $A\neq0$.  The restrictions of the five monomials $T_1,T_2,T_3,AB_1,AB_2$ to $\cS_1$ span the five-dimensional space of polynomials of degree $3$ and weight $-8$, with basis $a_1a_2^2$, $a_1^2a_3$, $a_0a_2a_3$, $a_0a_1a_4$, $a_0^2a_5$, and the restrictions of $N$ and $A^2$ span the space with basis $a_1^2$ and $a_0a_2$.  Solving the two linear systems gives \cref{eq:V-numerators}, and \cref{eq:V-restrictions} follows on dividing by $A^2|_{\cS_1}=\tfrac{4}{2025}a_1^2$.  These are identities of polynomials in $a_0,\ldots,a_5$, verified as in \cref{app:slice-ring-lifts}.

The invariant field of a torus acting diagonally on $V_5$ is generated by the Laurent monomials of weight zero: a rational invariant is a quotient of two semi-invariant polynomials of the same weight, and dividing both by a monomial of that weight expresses it through weight-zero Laurent monomials.  The weight-zero Laurent monomials form the lattice $\{e\in\Z^6:\sum(3-i)e_i=0\}$ of rank $5$, and the field they generate is generated by any lattice basis.  The exponent vectors of the five functions in \cref{eq:V-restrictions} are
\[
 (0,0,0,1,0,0),\ (0,-1,2,0,0,0),\ (1,-1,0,0,1,0),\ (2,-2,0,0,0,1),\
 (1,-2,1,0,0,0),
\]
whose $5\times5$ minors have greatest common divisor $1$, so they form a basis of that lattice, and $k(V_5)^{\Gamma_1}=k(v_1,\ldots,v_5)$ with $v_j=V_j|_{\cS_1}$.  Since $k(Y)=k(V_5)^{\Gamma_1}$ has transcendence degree $5=\dim Y$, the $v_j$ are algebraically independent.

For separation, let $F,F'\in\cS_1$ have $a_1,a_1'\neq0$ and equal $v_j$.  The condition $D=0$, $A\neq0$ places both maps in $Y^\circ$, so it suffices to show that $F$ and $F'$ lie in the same $\Gamma_1$-orbit.  Acting by $s$ with $s^4a_1=1$ normalizes $a_1=1$.  The residual group is $\mu_4$, which acts through $s^2=\pm1$ by $(a_0,a_2,a_4)\mapsto\pm(a_0,a_2,a_4)$ and fixes $a_3,a_5$.  With $a_1=a_1'=1$ the hypotheses read $a_3=a_3'$, $a_2^2=a_2'^2$, $a_0a_4=a_0'a_4'$, $a_0^2a_5=a_0'^2a_5'$, $a_0a_2=a_0'a_2'\neq0$.  Hence $a_2'=\pm a_2$.  After applying the residual sign we may assume $a_2'=a_2$, and then $a_0'=a_0$, $a_4'=a_4$, $a_5'=a_5$.  Thus $F'=F$.
\end{proof}

The hypothesis $V_5\neq0$ is exact.  If $a_0=0$ the values \cref{eq:V-restrictions} are $(a_3,a_2^2/a_1,0,0,0)$ and do not see $a_4,a_5$.  If $a_2=0$ they are $(a_3,0,a_0a_4/a_1,a_0^2a_5/a_1^2,0)$ and are constant along the curve $(ta_0,a_1,0,a_3,t^{-1}a_4,t^{-2}a_5)$, whose points with $a_1=1$ are conjugate only for $t'=\pm t$.

The ring-level statement parallels \cref{thm:cubic-chart-ring-generators}.  On the open set $a_1\neq0$ of $\cS_1$, normalization to $a_1=1$ identifies the $\Gamma_1$-orbits with the orbits of $\mu_2$ acting by sign on $(a_0,a_2,a_4)$ and trivially on $(a_3,a_5)$, so the invariant ring of $\Gamma_1$ on $\{a_1\neq0\}$ is generated by the eight functions
\begin{equation}\label{eq:eight-D-zero}
 a_3,\quad \frac{a_2^2}{a_1},\quad \frac{a_0a_4}{a_1},\quad
 \frac{a_0a_2}{a_1^2},\quad \frac{a_0^2}{a_1^3},\quad
 a_1a_5,\quad a_2a_4,\quad a_1a_4^2,
\end{equation}
and these separate all $\Gamma_1$-orbits with $a_1\neq0$.  The first five are the restrictions of $V_1,V_2,V_3,V_5$ and of $-\tfrac{81}{250}L^2/A^3$.  The last three are restrictions of elements of $k[\cM_4^1][A^{-1}]$, since $Y\cap\{A\neq0\}$ is closed in the affine variety $\cM_4^1\cap\{A\neq0\}$ and restriction is surjective. The lowest bidegrees in which $a_1^{k+1}a_5$ and $a_1^ka_2a_4$ occur as restrictions of invariants are $(6,10)$, and we do not display these expressions.  Consequently two maps with $D=0$ and $A\neq0$ are conjugate if and only if the eight functions \cref{eq:eight-D-zero} agree, and \cref{cor:separation-moduli} remains the statement of record on the complement, which consists of the classes in $Y$ with $a_1=0$ on the slice, the classes whose cubic is a cube, and the classes with $G=0$.  The last are the classes of binary quintics with nonzero discriminant, which are separated by the weighted point $[c_{4,0}:c_{8,0}:c_{12,0}:c_{18,0}]$.

\subsection{An algorithm for conjugacy}\label{subsec:deciding-conjugacy}

We collect the separation statements into an algorithm which decides, for two degree-four rational maps given by their coefficients, whether they are conjugate.  An implementation is described in \cref{app:implementation}.  All functions involved are rational functions with coefficients in $\Q$ of the ten coefficients of $(f_0,f_1)$: the pair $(F,G)$ is linear in them by \cref{eq:cg-map}, every $J_i$ is a polynomial in the coefficients of $(F,G)$, and every absolute invariant is a quotient of such polynomials.

For $f_0=\sum_{i=0}^4a_ix^iy^{4-i}$ and $f_1=\sum_{i=0}^4b_ix^iy^{4-i}$ one has
\[
\begin{split}
 F&=a_0y^5+(a_1-b_0)xy^4+(a_2-b_1)x^2y^3+(a_3-b_2)x^3y^2+(a_4-b_3)x^4y-b_4x^5,\\
 G&=(a_1+4b_0)y^3+(2a_2+3b_1)xy^2+(3a_3+2b_2)x^2y+(4a_4+b_3)x^3.
\end{split}
\]
Write $G=\sum_{j=0}^3g_jx^jy^{3-j}$ and $F=\sum_{i=0}^5\alpha_ix^iy^{5-i}$.  Then $D=\tfrac{2}{27}\disc(G)$, where $\disc(G)=g_1^2g_2^2-4g_0g_2^3-4g_1^3g_3+18g_0g_1g_2g_3 -27g_0^2g_3^2$, and the first coordinate is
\(
 U_1=-\frac{140\,S}{\disc(G)},
\)
where
\[
\begin{split}
 S={}&15\alpha_0g_1g_3^2-5\alpha_0g_2^2g_3-9\alpha_1g_0g_3^2-2\alpha_1g_1g_2g_3     +\alpha_1g_2^3 +6\alpha_2g_0g_2g_3+\alpha_2g_1^2g_3-\alpha_2g_1g_2^2\\
    &-6\alpha_3g_0g_1g_3-\alpha_3g_0g_2^2+\alpha_3g_1^2g_2+9\alpha_4g_0^2g_3+2\alpha_4g_0g_1g_2-\alpha_4g_1^3-15\alpha_5g_0^2g_2+5\alpha_5g_0g_1^2 .
\end{split}
\]

The reconstruction by the six coordinates is valid on $\delta\neq0$, that is, where $U_2^3/2-27U_3^2\neq0$.  On the locus $\delta=0$ one must use additional regular invariants.  For later use we also record a smaller generating system on the $\Gamma$-stable hyperplane $\eta=0$.

\begin{proposition}
\label{prop:eta-zero-separation}
Let $\varphi$ and $\psi$ be degree-four rational maps with $D\neq0$ whose slice representatives satisfy $\eta=0$.  Then $\varphi$ and $\psi$ are conjugate if and only if the eight functions
\[
 U_1,\ U_2,\ U_3,\ U_4,\ U_5,\  \frac{Q_{p_{02}}}{D^{3}},\ \frac{Q_{p_{12}}}{D^{4}},\ \frac{Q_{p_{03}}}{D^{4}}
\]
take the same values at $\varphi$ and $\psi$.
\end{proposition}

\begin{proof}
The hyperplane $\{\eta=0\}\subset V_5$ is $\Gamma$-stable, since $\eta$ transforms by the sign character.  For a finite group in characteristic zero, restriction of invariants to a stable closed subvariety is surjective, by averaging over the group.  Hence $k[V_5]^\Gamma\to k[\{\eta=0\}]^\Gamma$ is surjective, and by \cref{thm:cubic-chart-ring-generators} the target is generated by the restrictions of the fourteen functions \cref{eq:fourteen-slice-generators}.  The six functions $e_2,w_2,w_{30},w_{21},w_{12},w_{03}$ contain the factor $\eta$ and restrict to zero.  The remaining eight, $t,p_{20},p_{11},p_{02},p_{30},p_{21},p_{12},p_{03}$, therefore generate $k[\{\eta=0\}]^\Gamma$, and invariants of a finite group separate its orbits.  By \cref{eq:field-generators-inside-fourteen} and \cref{eq:extra-slice-lift-identities} these eight are the restrictions of the displayed functions, and \cref{lem:D-chart-isomorphism} identifies $\Gamma$-orbits with conjugacy classes.
\end{proof}

On the locus $\delta=0$ two of the $x_i$ coincide, the system \cref{eq:y-reconstruction} has rank $2$, and the fibres of $(U_1,\ldots,U_6)$ are positive-dimensional.  There we know no subset of the fourteen functions of \cref{cor:cubic-chart-joint-generators} which suffices, and \cref{cor:separation-D-chart} is used with all fourteen.

\Cref{alg:conjugacy} decides conjugacy.  Its input is two pairs of quartic forms with nonzero resultant, and its output is a truth value.  It uses the following subroutines, all of which are exact polynomial evaluations: $\textsc{Pair}(f_0,f_1)$ returns $(F,G)=\Theta_4(f_0,f_1)$ by \cref{eq:associated-forms-intro}, $\textsc{Disc}(G)$ returns $D=(h,h)_2$, $\textsc{Coord}(F,G)$ returns $(U_1,\ldots,U_6)$ by \cref{eq:absolute-U}, $\textsc{Chart}(F,G)$ returns the fourteen values \cref{eq:fourteen-descended-generators}, $\textsc{Div}(F,G)$ returns $(V_1,\ldots,V_5)$ by \cref{eq:V-coordinates}, and $\textsc{Abs}(F,G)$ returns the $5411$ values $J^a\Delta_{F,G}^{-m}$, $(a,m)\in\Hilb(\Sigma)$, of \cref{thm:moduli-coordinate-ring}.  Every quantity compared is an absolute invariant: on the chart $D\neq0$ these are the six coordinates $U_i$ and, when $U_2^3/2-27U_3^2=0$, the fourteen generators of $\cA_D$.  On the divisor $D=0$ they are the five coordinates $V_j$ of \cref{thm:V-coordinates}, and on the remaining locus they are the generators of $k[\cM_4^1]$.

\begin{algorithm}[!ht]
\caption{Conjugacy of two degree-four rational maps}
\label{alg:conjugacy}
\begin{algorithmic}[1]
\Require $f_0,f_1,g_0,g_1\in V_4$ with $\Res(f_0,f_1)\neq0$ and
 $\Res(g_0,g_1)\neq0$
\Ensure \textsc{true} if $[f_0:f_1]$ and $[g_0:g_1]$ are conjugate over
 $\bar k$, \textsc{false} otherwise
\State $(F,G)\gets\textsc{Pair}(f_0,f_1)$;\quad
       $(F',G')\gets\textsc{Pair}(g_0,g_1)$
\State $D\gets\textsc{Disc}(G)$;\quad $D'\gets\textsc{Disc}(G')$
\If{exactly one of $D,D'$ is zero}
  \State \Return \textsc{false}
\EndIf
\If{$D\neq0$}
  \Comment{cubic-discriminant chart}
  \State $(U_1,\ldots,U_6)\gets\textsc{Coord}(F,G)$;\quad
         $(U_1',\ldots,U_6')\gets\textsc{Coord}(F',G')$
  \If{$(U_1,\ldots,U_6)\neq(U_1',\ldots,U_6')$}
    \State \Return \textsc{false}
    \Comment{forward direction of \cref{thm:explicit-absolute-invariants}}
  \EndIf
  \If{$U_2^3/2-27U_3^2\neq0$}
    \State \Return \textsc{true}
    \Comment{\cref{cor:U-separation}}
  \EndIf
  \State \Return $\bigl[\textsc{Chart}(F,G)=\textsc{Chart}(F',G')\bigr]$
  \Comment{\cref{cor:separation-D-chart}}
\EndIf
\State $A\gets J^{(1)}_{1,3}(F,G)$;\quad $A'\gets J^{(1)}_{1,3}(F',G')$
       \Comment{$D=D'=0$: the divisor $Y$}
\If{exactly one of $A,A'$ is zero}
  \State \Return \textsc{false}
\EndIf
\If{$A\neq0$}
  \State $(V_1,\ldots,V_5)\gets\textsc{Div}(F,G)$;\quad
         $(V_1',\ldots,V_5')\gets\textsc{Div}(F',G')$
  \If{$(V_1,\ldots,V_5)\neq(V_1',\ldots,V_5')$}
    \State \Return \textsc{false}
  \EndIf
  \If{$V_5\neq0$}
    \State \Return \textsc{true}
    \Comment{\cref{thm:V-coordinates}}
  \EndIf
\EndIf
\State \Return $\bigl[\textsc{Abs}(F,G)=\textsc{Abs}(F',G')\bigr]$
  \Comment{\cref{cor:separation-moduli}}
\end{algorithmic}
\end{algorithm}


Each return statement is justified by the result named in its comment.  The two returns of \textsc{false} after comparing coordinates use only the forward direction, that conjugate maps have equal absolute invariants.  In the return $\textsc{Chart}(F,G)=\textsc{Chart}(F',G')$ the fourteen values may be replaced by the eight of \cref{prop:eta-zero-separation} when both maps have $\eta=0$.  The last return compares $5411$ rational numbers.  By \cref{cor:separation} it is equivalent to the finite test $J_j'=\lambda^{d_j}J_j$ for all $j$ and some $\lambda$ with $\lambda^{d_{i_0}}=J'_{i_0}/J_{i_0}$, $J_{i_0}\neq0$, which needs at most $d_{i_0}$ trials.  Every step is a finite exact computation over the field generated by the coefficients, and by \cref{cor:rationality-over-Q} the criteria hold over any field of characteristic zero, conjugacy being understood over its algebraic closure.

\section{Conclusions and further directions}
\label{sec-7}

The three layers of the invariant-theoretic program are now explicit for $V_5\oplus V_3$.

On the polynomial layer, \cref{thm:minimal-generators} determines a minimal generating set of fifty invariants, so $\beta(\cR_{5,3})=18$ and the Gordan bound $42$ of \cref{thm:gordan-degree-bound} is far from sharp.  \Cref{thm:hsop} exhibits a homogeneous system of parameters over which $\cR_{5,3}$ is free of rank $604$, and by \cref{thm:no-bihomogeneous-hsop} some parameter of any such system fails to be bihomogeneous.

On the field layer, \cref{thm:explicit-absolute-invariants} and \cref{cor:rationality-over-Q} give six birational coordinates defined over $\Q$, and \cref{cor:U-separation} shows that they separate conjugacy classes where $U_2^3/2-27U_3^2\neq0$.

On the ring layer, \cref{thm:cubic-chart-ring-generators} and \cref{cor:cubic-chart-joint-generators} determine the cubic-discriminant chart, and \cref{thm:moduli-coordinate-ring} gives $5411$ explicit generators of $k[\cM_4^1]$.  By \cref{cor:no-parameter-generation} this ring is never generated by quotients of a system of parameters.

\Cref{tab:degree-comparison} compares the three smallest degrees.

\begin{table}[ht]
\caption{The joint invariant ring of degree-$d$ rational maps.}
\label{tab:degree-comparison}
\centering
\begin{tabular}{ccccc}
\toprule
$d$&pair&$\dim\cR$&$\dim\cM_d^1$&generators\\
\midrule
$2$&$V_3\oplus V_1$&$3$&$2$&$3$\\
$3$&$V_4\oplus V_2$&$5$&$4$&$6$\\
$4$&$V_5\oplus V_3$&$7$&$6$&$50$\\
\bottomrule
\end{tabular}
\end{table}

Three problems remain.  The first is the ideal of relations among the fifty generators.  It is what is needed for a minimal generating set of $k[\cM_4^1]$ and for the secondary invariants of \cref{cor:hironaka}.

The second is the dynamical interpretation of the coordinates $U_1,\ldots,U_6$.  By \cref{prop:multiplier} the multipliers at the fixed points are rational functions of $(F,G)$, so the classical multiplier invariants are elements of $\Frac(\cR_{5,3})_0$.  Their relation to $U_1,\ldots,U_6$, the loci where they fail to separate orbits, and the automorphism stratification of $\cM_4^1$ remain to be determined.
The third is arithmetic.  By \cref{cor:rationality-over-Q} all the invariants are defined over $\Q$, so fields of moduli, twists, and reduction can be treated explicitly. We intend to address some of these in future work.

\appendix

\crefalias{section}{appendix}
\crefalias{subsection}{appendix}

\section{Primitive order patterns for Gordan's theorem}
\label{app:gordan-patterns}

The following table is the complete uncolored primitive list used in \cref{thm:gordan-degree-bound}.  Repetition is written exponentially: for example, $1^3=3$ means three order-one quintic covariants matched with one order-three cubic covariant.  The column $\widehat d$ is the largest possible total coefficient degree after the repeated orders are replaced by named covariants, using \cref{eq:max-degree-by-order}.  Thus it is an upper bound for every colored Gordan transvectant belonging to that pattern.

For completeness, the enumeration has the following short independent check.  For each pair $(u,v)\in\N^2$ with $s=2u+3v\leq51$, generate the partitions $\lambda$ of $s$ with parts in $\{1,2,3,4,5,6,7,9\}$ by the standard recurrence which either omits or adjoins one copy of the largest allowed part.  Let $\Sub(\lambda)$ be the set of sums of submultisets of $\lambda$.  The identity    \( \lambda=2^u3^v\)  is primitive exactly when
\[
 \Sub(\lambda)\cap    \{\,2u'+3v':0\leq u'\leq u,\ 0\leq v'\leq v\,\}  =\{0,s\}.
\]
This criterion produces precisely the rows below.  Together with the length bound used in \cref{thm:gordan-degree-bound}, it also proves that no further primitive identity is omitted.

\begingroup \footnotesize \setlength{\LTpre}{0.5em} \setlength{\LTpost}{0.5em}
\begin{longtable}{c r r @{\qquad} c r r}
\caption{Primitive order identities for $V_5\oplus V_3$.}
\label{tab:gordan-patterns}\\
\toprule
Order identity & Common order & $\widehat d$ & Order identity & Common order & $\widehat d$\\
\midrule
\endfirsthead
\toprule
Order identity & Common order & $\widehat d$ & Order identity & Common order & $\widehat d$\\
\midrule
\endhead
$2=2$ & $2$ & $10$ & $2\,7=3^{3}$ & $9$ & $22$ \\
$1^{2}=2$ & $2$ & $28$ & $2^{2}\,5=3^{3}$ & $9$ & $32$ \\
$3=3$ & $3$ & $12$ & $1\,4^{2}=3^{3}$ & $9$ & $34$ \\
$1\,2=3$ & $3$ & $24$ & $1^{2}\,7=3^{3}$ & $9$ & $40$ \\
$1^{3}=3$ & $3$ & $42$ & $5^{2}=2^{5}$ & $10$ & $24$ \\
$4=2^{2}$ & $4$ & $10$ & $3\,7=2^{5}$ & $10$ & $24$ \\
$1\,3=2^{2}$ & $4$ & $26$ & $1\,9=2^{2}\,3^{2}$ & $10$ & $26$ \\
$5=2\,3$ & $5$ & $12$ & $1\,9=2^{5}$ & $10$ & $26$ \\
$1\,4=2\,3$ & $5$ & $24$ & $4\,7=2\,3^{3}$ & $11$ & $22$ \\
$6=3^{2}$ & $6$ & $10$ & $5\,7=3^{4}$ & $12$ & $24$ \\
$6=2^{3}$ & $6$ & $10$ & $5\,7=2^{6}$ & $12$ & $24$ \\
$3^{2}=2^{3}$ & $6$ & $24$ & $4^{3}=3^{4}$ & $12$ & $30$ \\
$2\,4=3^{2}$ & $6$ & $20$ & $3\,9=2^{6}$ & $12$ & $24$ \\
$2^{3}=3^{2}$ & $6$ & $30$ & $2\,5^{2}=3^{4}$ & $12$ & $34$ \\
$1\,5=3^{2}$ & $6$ & $26$ & $1\,4\,7=3^{4}$ & $12$ & $36$ \\
$1\,5=2^{3}$ & $6$ & $26$ & $7^{2}=2\,3^{4}$ & $14$ & $24$ \\
$1^{2}\,4=3^{2}$ & $6$ & $38$ & $7^{2}=2^{7}$ & $14$ & $24$ \\
$7=2^{2}\,3$ & $7$ & $12$ & $5\,9=2^{7}$ & $14$ & $24$ \\
$1\,6=2^{2}\,3$ & $7$ & $24$ & $5^{3}=3^{5}$ & $15$ & $36$ \\
$4^{2}=2\,3^{2}$ & $8$ & $20$ & $4^{2}\,7=3^{5}$ & $15$ & $32$ \\
$3\,5=2^{4}$ & $8$ & $24$ & $1\,7^{2}=3^{5}$ & $15$ & $38$ \\
$1\,7=2\,3^{2}$ & $8$ & $26$ & $7\,9=2^{8}$ & $16$ & $24$ \\
$1\,7=2^{4}$ & $8$ & $26$ & $9^{2}=2^{9}$ & $18$ & $24$ \\
$9=3^{3}$ & $9$ & $12$ & $4\,7^{2}=3^{6}$ & $18$ & $34$ \\
$9=2^{3}\,3$ & $9$ & $12$ & $7^{3}=3^{7}$ & $21$ & $36$ \\
$4\,5=3^{3}$ & $9$ & $22$ & & & \\
\bottomrule
\end{longtable}
\endgroup

\subsection{The fifty generators}\label{app:generators}

The table below uses the covariants of \cref{eq:quintic-initial}, \cref{eq:quintic-covariants}, and \cref{eq:cubic-covariants}.  

\begin{small}
\begin{longtable}{c l @{\qquad} c l @{\qquad} c l}
\toprule
Invariant & transvectant & Invariant & transvectant & Invariant & transvectant \\
\midrule
\endhead
$J_{0,4}^{(1)}$ & $D$ & $J_{4,4}^{(1)}$ & $\bigl(c_{4,4},h^{2}\bigr)_{4}$ & $J_{6,4}^{(3)}$ & $\bigl(F i c_{3,5},G^{4}\bigr)_{12}$ \\
$J_{1,3}^{(1)}$ & $\bigl(F,G h\bigr)_{5}$ & $J_{4,4}^{(2)}$ & $\bigl(c_{4,6},G q\bigr)_{6}$ & $J_{7,3}^{(1)}$ & $\bigl(i c_{5,1},q\bigr)_{3}$ \\
$J_{2,2}^{(1)}$ & $\bigl(i,h\bigr)_{2}$ & $J_{4,4}^{(3)}$ & $\bigl(F^{2} i,G^{4}\bigr)_{12}$ & $J_{7,3}^{(2)}$ & $\bigl(c_{7,5},G h\bigr)_{5}$ \\
$J_{2,2}^{(2)}$ & $\bigl(H,G^{2}\bigr)_{6}$ & $J_{5,3}^{(1)}$ & $\bigl(c_{5,3},q\bigr)_{3}$ & $J_{7,3}^{(3)}$ & $\bigl(c_{3,5} c_{4,4},G^{3}\bigr)_{9}$ \\
$J_{3,1}^{(1)}$ & $\bigl(c_{3,3},G\bigr)_{3}$ & $J_{5,3}^{(2)}$ & $\bigl(F c_{4,4},G^{3}\bigr)_{9}$ & $J_{8,2}^{(1)}$ & $\bigl(c_{8,2},h\bigr)_{2}$ \\
$J_{4,0}^{(1)}$ & $c_{4,0}$ & $J_{5,3}^{(3)}$ & $\bigl(F i^{2},G^{3}\bigr)_{9}$ & $J_{8,2}^{(2)}$ & $\bigl(c_{3,5} c_{5,1},G^{2}\bigr)_{6}$ \\
$J_{1,5}^{(1)}$ & $\bigl(F,h q\bigr)_{5}$ & $J_{6,2}^{(1)}$ & $\bigl(c_{6,2},h\bigr)_{2}$ & $J_{9,1}^{(1)}$ & $\bigl(c_{9,3},G\bigr)_{3}$ \\
$J_{2,4}^{(1)}$ & $\bigl(H,G q\bigr)_{6}$ & $J_{6,2}^{(2)}$ & $\bigl(i c_{4,4},G^{2}\bigr)_{6}$ & $J_{9,1}^{(2)}$ & $\bigl(i c_{7,1},G\bigr)_{3}$ \\
$J_{3,3}^{(1)}$ & $\bigl(c_{3,3},q\bigr)_{3}$ & $J_{6,2}^{(3)}$ & $\bigl(i^{3},G^{2}\bigr)_{6}$ & $J_{9,3}^{(1)}$ & $\bigl(i c_{7,1},q\bigr)_{3}$ \\
$J_{3,3}^{(2)}$ & $\bigl(c_{3,5},G h\bigr)_{5}$ & $J_{7,1}^{(1)}$ & $\bigl(i c_{5,1},G\bigr)_{3}$ & $J_{10,2}^{(1)}$ & $\bigl(c_{5,1}^{2},h\bigr)_{2}$ \\
$J_{3,3}^{(3)}$ & $\bigl(t,G^{3}\bigr)_{9}$ & $J_{8,0}^{(1)}$ & $c_{8,0}$ & $J_{11,1}^{(1)}$ & $\bigl(c_{5,1} c_{6,2},G\bigr)_{3}$ \\
$J_{4,2}^{(1)}$ & $\bigl(c_{4,6},G^{2}\bigr)_{6}$ & $J_{3,7}^{(1)}$ & $\bigl(F^{3},G^{4} q\bigr)_{15}$ & $J_{12,0}^{(1)}$ & $c_{12,0}$ \\
$J_{5,1}^{(1)}$ & $\bigl(c_{5,3},G\bigr)_{3}$ & $J_{4,6}^{(1)}$ & $\bigl(F^{2} i,G^{3} q\bigr)_{12}$ & $J_{12,2}^{(1)}$ & $\bigl(c_{5,1} c_{7,1},h\bigr)_{2}$ \\
$J_{2,6}^{(1)}$ & $\bigl(H,q^{2}\bigr)_{6}$ & $J_{5,5}^{(1)}$ & $\bigl(F c_{4,4},G^{2} q\bigr)_{9}$ & $J_{13,1}^{(1)}$ & $\bigl(c_{6,2} c_{7,1},G\bigr)_{3}$ \\
$J_{3,5}^{(1)}$ & $\bigl(c_{3,5},h q\bigr)_{5}$ & $J_{5,5}^{(2)}$ & $\bigl(F^{2} c_{3,5},G^{5}\bigr)_{15}$ & $J_{15,1}^{(1)}$ & $\bigl(c_{7,1} c_{8,2},G\bigr)_{3}$ \\
$J_{3,5}^{(2)}$ & $\bigl(t,G^{2} q\bigr)_{9}$ & $J_{6,4}^{(1)}$ & $\bigl(c_{6,4},h^{2}\bigr)_{4}$ & $J_{18,0}^{(1)}$ & $c_{18,0}$ \\
$J_{3,5}^{(3)}$ & $\bigl(F^{3},G^{5}\bigr)_{15}$ & $J_{6,4}^{(2)}$ & $\bigl(i c_{4,4},G q\bigr)_{6}$ &  &  \\
\bottomrule
\end{longtable}
\end{small}

The first subscript pair on $J_{a,b}^{(r)}$ is its bidegree in the coefficients of $(F,G)$; the superscript distinguishes listed invariants of the same bidegree.

\subsection{The rational-map resultant in the generators}
\label{app:resultant}

The invariant $\Delta_{F,G}$ of \cref{def:modular-resultant} lies in $\cR_8$ and is the following polynomial in the generators:
\begin{footnotesize}
\begin{equation}\label{eq:resultant-in-generators}
\begin{split}
\Delta_{F,G}
&=125\Bigl(
 \frac{1830}{49}J_{0,4}^{(1)}J_{3,1}^{(1)}
 -\frac{50625}{1232}J_{0,4}^{(1)}J_{4,0}^{(1)}
 -\frac{180}{7}J_{1,3}^{(1)}J_{2,2}^{(1)}
 +\frac{225}{14}J_{1,3}^{(1)}J_{2,2}^{(2)}
 -\frac{4875}{22}J_{1,3}^{(1)}J_{3,1}^{(1)}
 +\frac{6075}{88}\bigl(J_{2,2}^{(1)}\bigr)^2 \\
&\qquad
 -\frac{10125}{44}J_{2,2}^{(1)}J_{2,2}^{(2)}
 -\frac{625}{2}J_{2,2}^{(1)}J_{3,1}^{(1)}
 +\frac{625}{4}\bigl(J_{2,2}^{(2)}\bigr)^2
 -\frac{6250}{3}J_{2,2}^{(2)}J_{3,1}^{(1)}
 +\frac{625}{2}J_{2,2}^{(2)}J_{4,0}^{(1)}
 +7500\bigl(J_{3,1}^{(1)}\bigr)^2 \\
&\qquad
 +\frac{3125}{4}\bigl(J_{4,0}^{(1)}\bigr)^2
 -\frac{8775}{154}J_{3,5}^{(1)}
 -\frac{2700}{91}J_{3,5}^{(2)}
 -J_{3,5}^{(3)}
 -\frac{1125}{4}J_{4,4}^{(1)}
 +\frac{5625}{11}J_{4,4}^{(2)}
 -\frac{75}{2}J_{4,4}^{(3)} \\
&\qquad
 +\frac{9375}{7}J_{5,3}^{(1)}-250J_{5,3}^{(2)}-375J_{5,3}^{(3)}
 -6750J_{6,2}^{(1)}-3750J_{6,2}^{(2)}-625J_{6,2}^{(3)}-50000J_{8,0}^{(1)}
 \Bigr).
\end{split}
\end{equation}
\end{footnotesize}
The identity was proved by exact evaluation at the $130$ rational points of the proof of \cref{thm:minimal-generators}: both sides lie in $\cR_8$, evaluation is injective on $\cR_8$ because the evaluation matrix of the $36$ monomials spanning $\cR_8$ has rank $36$ over $\Q$, and both sides agree at all $130$ points.

\subsection{Descent data for the cubic-discriminant chart}
\label{app:slice-ring-lifts}

We use the abbreviations in \cref{eq:absolute-shorthand}.  For   \( g=e_2\), \(p_{02}\), \(p_{12}\), \(p_{03}\), \(w_2\), \(w_{21}\), \(w_{12}\), \(w_{03}\)
define $Q_g$ by the following formulas:
\begin{footnotesize}
\begin{equation}\label{eq:Q}
\begin{split}
Q_{e_2}
&=264600L^2, \\[0.25em]
Q_{p_{02}}
&=\frac1{30}\Bigl(4083210L^2+D\bigl(2453760N+581405A^2+D(2687080B_2-835416B_1)\bigr)\Bigr), \\[0.25em]
Q_{p_{12}}
&=\frac{5}{756756}\Bigl(133383920670AL^2
 +D\bigl(4587439736880LM+1869573265560AN+407613761555A^3 \\
&\qquad
 +D(12444992560T_3-1305564986880T_2-1725889117680T_1+2239075899060AB_2 \\
&\qquad
 -1153318029864AB_1-289845437024DR)\bigr)\Bigr), \\[0.25em]
Q_{p_{03}}
&=\frac1{252252}\Bigl(1459163675970AL^2
 -D\bigl(7553317852080LM+2899790613720AN+648701088035A^3 \\
&\qquad
 +D(370860778288T_3-3781089607680T_2-1987532754480T_1+2085082539540AB_2 \\
&\qquad
 -1274048816040AB_1-459131079264DR)\bigr)\Bigr), \\[0.25em]
Q_{w_2}
&=49392000L(DM-AL), \\[0.25em]
Q_{w_{21}}
&=\frac{625}{9009}L\Bigl(
 438837128730L^3+982935233280DLN+215194604625DA^2L-83523440D^2W \\
&\qquad
 +526319806320D^2LB_2-178109310456D^2LB_1-22314544560D^2AM \\
&\qquad
 -8330894880D^3S_3+75289456620D^3S_2+99761875488D^3S_1
 \Bigr), \\[0.25em]
Q_{w_{12}}
&=\frac{25}{3003}L\Bigl(
 -2832494194530L^3-4448817172800DLN-1081547792325DA^2L+12444992560D^2W \\
&\qquad
 -3200034111120D^2LB_2+914989205592D^2LB_1-8471509200D^2AM \\
&\qquad
 +651961944480D^3S_3-291086990460D^3S_2-664047594144D^3S_1
 \Bigr), \\[0.25em]
Q_{w_{03}}
&=\frac1{1001}L\Bigl(
 18282462528330L^3+16479966222720DLN+3176968579785DA^2L \\
&\qquad
 -1854303891440D^2W+18276219284400D^2LB_2-5198207741496D^2LB_1 \\
&\qquad
 +4876420626000D^2AM-9330462224160D^3S_3+1385837078220D^3S_2+3660996971808D^3S_1
 \Bigr).
\end{split}
\end{equation}
\end{footnotesize}
Their bidegrees and denominator powers are
\begin{equation}\label{eq:extra-lift-degree-table}
\begin{array}{c|cccccccc}
g&e_2&p_{02}&p_{12}&p_{03}&w_2&w_{21}&w_{12}&w_{03}\\ \hline
m_g&3&3&4&4&4&6&6&6\\
\bideg(Q_g)
 &(2,10)&(2,10)&(3,13)&(3,13)&(3,13)
 &(4,20)&(4,20)&(4,20).
\end{array}
\end{equation}
With $D(G_0)=2/27$, direct expansion gives the identities \cref{eq:extra-slice-lift-identities}.  An exact finite verification can be made without random sampling: evaluate both sides on the barycentric simplex lattice of degree $d=2,3$, or $4$ in the six coefficients of $F$.  This set is unisolvent for homogeneous polynomials of degree $d$, and rational Gaussian elimination proves equality.  Thus the displayed identities are polynomial identities in the coefficients of $F$.

\subsection{Spanning certificates}\label{app:certificates}

The generation step of \cref{thm:minimal-generators} rests on the nonvanishing of finitely many determinants.  These determinants are documented by two certificate files, distributed with this paper as ancillary files:
\[
\texttt{degree-30-certificate.json},\qquad \texttt{gordan-six-degrees-certificate.json}.
\]
The first file certifies the $242$ bidegrees with even total degree at most $30$ and $\dim_k\cR_{a,b}>0$, with $p=65521$; the certified dimensions sum to $31638$.  The second file certifies the $51$ bidegrees of the $81$ substituted Gordan patterns of total degrees $32$ to $42$, with $p=251$; the certified dimensions sum to $34146$.

Each file is a JSON document.  The field \texttt{prime} records $p$.  The field \texttt{evaluation\_points} records the coordinates of the evaluation points, namely the six coefficients of $F$ and the four coefficients of $G$ modulo $p$ for each point; the first file records $815$ points, the second records $2184$ points, and the certificate of a bidegree of dimension $n$ uses the first $n$ points of the list.  The field \texttt{point\_seed} records the seed $20260815$ of the deterministic generator that produced the points.  The field \texttt{source\_sha256} records the SHA-256 hashes of the scripts that built the file.  The field \texttt{records} contains one record per certified bidegree.  A record contains the bidegree $(a,b)$, the dimension $n=\dim_k\cR_{a,b}$, a list of $n$ monomials in the fifty generators, each encoded as the hexadecimal string of its $50$ exponent bytes, and the determinant of the $n\times n$ evaluation matrix modulo $p$.  The ordering of the fifty generators underlying the exponent vectors is fixed in \texttt{gordan\_rank.py} and is the same for building and for verification.

The verification procedure is the following.  The commands
\begin{verbatim}
python3 low_degree_certificate.py \
        --verify degree-30-certificate.json
python3 six_degree_certificate.py \
        --verify gordan-six-degrees-certificate.json
\end{verbatim}
recompute the values of the fifty generators at the recorded points by iterated application of \cref{eq:transvectant} modulo $p$, recompute every dimension from \cref{eq:bigraded-dimension}, form the recorded square matrix of every record, and recompute every determinant by blocked Gaussian elimination modulo $p$.  Verification fails unless every recomputed dimension and every recomputed determinant agrees with the record, the set of certified bidegrees equals the target family, and, for the second file, every substitution count agrees with an independent recount.  A successful verification establishes the nonvanishing of every recorded determinant, and this nonvanishing is the only computational input to the generation step.

The supporting modules are \texttt{gordan\_audit.py}, which enumerates the primitive patterns of \cref{tab:gordan-patterns}, counts their substitutions by bidegree, and computes the dimensions \cref{eq:bigraded-dimension}, and \texttt{gordan\_rank.py}, which implements modular transvection, the enumeration of the monomials of a bidegree, and the elimination routines.  The enumeration in \texttt{gordan\_audit.py} confirms independently that no substituted pattern has total degree exceeding $42$, in agreement with \cref{thm:gordan-degree-bound}.

\subsection{Implementation}\label{app:implementation}

The file \texttt{deg4\_conjugacy.py}, distributed with this paper as an ancillary file, implements the objects of the paper in exact arithmetic with SymPy.  It contains the transvectant \cref{eq:transvectant}, the Clebsch--Gordan pair and its inverse \cref{eq:quartic-reconstruction}, the covariants of \cref{sec:covariants}, the fifty generators of \cref{app:generators}, the rational-map resultant $\Delta_{F,G}$, the numerators $P_1,\ldots,P_6$ and the coordinates $U_1,\ldots,U_6$ of \cref{thm:explicit-absolute-invariants}, the eight functions $Q_g/D^{m_g}$ of \cref{app:slice-ring-lifts}, the coordinates $V_1,\ldots,V_5$ of \cref{thm:V-coordinates}, the $5411$ generators of \cref{thm:moduli-coordinate-ring}, and \cref{alg:conjugacy} as the function \texttt{conjugate(f0,\,f1,\,g0,\,g1)}, which returns a truth value together with the case of the algorithm that decided it.

The certificate files and the scripts of \cref{app:certificates} are distributed in the same package.

The script was used to check the following statements independently of the computations in the proof of \cref{thm:minimal-generators}: the identity of \cref{app:resultant} at random rational points; the six identities \cref{eq:slice-restriction-identities} and the eight identities \cref{eq:extra-slice-lift-identities} as polynomial identities in the six coefficients of $F$; the relation $u_6=w_{30}$ of \cref{eq:field-generators-inside-fourteen}; the identities \cref{eq:V-restrictions} on the double-root slice.

\bibliographystyle{elsarticle-num}
\bibliography{references}

\end{document}